\documentclass[11pt,leqno,a4paper]{article}

\usepackage{xcolor}
\usepackage{color}
\usepackage{mathrsfs} 

\usepackage{amsthm,amsfonts,amssymb,amsmath,oldgerm}
\usepackage{amsthm,amsfonts,amssymb,amsmath}
\usepackage{epsfig}
\usepackage[thinlines]{easybmat}
\usepackage{epstopdf}
\usepackage[utf8]{inputenc}
\usepackage[T1]{fontenc}
\usepackage{textcomp}

\DeclareMathOperator{\Ai}{Ai}
\newtheorem{theorem}{Theorem}[section]
\newtheorem{proposition}[theorem]{Proposition}

\theoremstyle{remark}
\newtheorem{remark}[theorem]{Remark}
\theoremstyle{definition}

{\everymath{\displaystyle\everymath{}}\array}%
{\endarray}

\def\dsp{\displaystyle}

\def\R{\mathbb{R}}

\def\d{\partial}

\def\v{\times}

\def\1{{\bf 1}}
\def\dx{\delta x}
\def\dsp{\displaystyle}

\title{Discrete transparent boundary conditions\\ for the mixed KDV-BBM equation}
\author{Christophe Besse\footnote{Institut de Math\'ematiques de
    Toulouse; UMR5219, Universit\'e de Toulouse; CNRS, UPS IMT, F-31062 Toulouse Cedex 9, France. \newline
\texttt{Email: christophe.besse@math.univ-toulouse.fr}},   Pascal Noble\footnote{Institut de Math\'ematiques de
    Toulouse; UMR5219, Universit\'e de Toulouse; CNRS, INSA, F-31077 Toulouse, France.  \newline
\texttt{Email: pascal.noble@math.univ-toulouse.fr}}, David Sanchez\footnote{Institut de Math\'ematiques de
    Toulouse; UMR5219, Universit\'e de Toulouse; CNRS, INSA, F-31077 Toulouse, France.\texttt{Email: david.sanchez@math.univ-toulouse.fr}}}

\begin{document}
\maketitle
\normalsize

\begin{abstract}
In this paper, we consider artificial boundary conditions for the linearized mixed Korteweg-de Vries (KDV) Benjamin-Bona-Mahoney (BBM) equation which models water waves in the small amplitude, large wavelength regime.  Continuous (respectively discrete ) artificial boundary conditions involve non local operators in time which in turn requires to compute time convolutions and invert the Laplace transform of an analytic function (respectively the $\mathcal{Z}$-transform of an holomorphic function).  In this paper, we propose a new, stable and fairly general strategy to carry out this crucial step in the design of transparent boundary conditions. For large time simulations, we also introduce a methodology based on the asymptotic expansion of coefficients involved in exact direct transparent boundary conditions. We illustrate the accuracy of our methods for Gaussian and wave packets initial data. 

\end{abstract}


\section{Introduction}

The Korteweg-de Vries (KdV) equation is a classical nonlinear, dispersive equation which models the unidirectional propagation of weakly nonlinear, long waves in the presence of dispersion. It is written
\begin{equation}\label{kdvl}
\displaystyle
\partial_t u+\partial_x u+\frac{3\varepsilon}{2}u\partial_x u+\frac{\mu}{6}\partial_{xxx}u=0,\quad \forall t>0,\quad \forall x\in\mathbb{R},
\end{equation}
\noindent
where $\varepsilon>0$ is the non linearity parameter, $\mu$ the shallowness/dispersion parameter and $\varepsilon,\mu$ have the same order (see \cite{L} for more details on the derivation of this particular equation). 
An alternative model which possesses better dispersive properties is obtained by noticing that, as $\varepsilon, \mu\to 0$, $\partial_x u=-\partial_t u+O(\varepsilon+\mu)$.  Then, one can substitute a time derivative to a spatial derivative in the dispersive term and (\ref{kdvl}) is asymptotically equivalent to

\begin{equation}\label{kdv-bbm}
\displaystyle
\partial_t\left(u-\alpha\partial_{xx} u\right)+\partial_x u+\frac{3\varepsilon}{2}u\partial_x u+(\frac{\mu}{6}-\alpha)\partial_{xxx}u=0,\quad \forall\: 0<\alpha\leq\frac{\mu}{6}.
\end{equation}

\noindent
If $\alpha=\mu/6$, the resulting equation is known as the Benjamin-Bona-Mahoney equation (BBM). We will denote (KdV-BBM) the mixed model (\ref{kdv-bbm}).  Both the (KdV) and the (KdV-BBM) equation possess solitary waves and cnoidal (periodic) waves solutions and it is of particular interest to determine the interaction between these waves and if these particular waves play a role in the description of the solutions of (\ref{kdvl}) or (\ref{kdv-bbm}) for asymptotically large time.

Indeed, in the limit of large scale and large time, the (KdV) equation is seen as a dispersive regularization of the Burgers equation 
$$
\displaystyle
\partial_t u+\partial_x u+\frac{3\varepsilon}{2}u\partial_x u=0.
$$
Dispersive regularization of hyperbolic conservation laws is known to generate so-called dispersive shock waves (DSW). In contrast to their diffusive counterparts, dispersive shocks have an oscillatory structure and expand with time so that the Rankine-Hugoniot jump conditions are not satisfied. There is a huge literature on these particular patterns for the Korteweg-de Vries equation. The numerical simulation of such patterns is a hard task: usually, such equations are solved by using spectral techniques which are particularly suitable to describe oscillatory phenomena but suppose that periodic boundary conditions are imposed to the edges of the computational domain. Moreover, due to the fact that the oscillatory part of the DSW expands in time, one has to take larger and larger computational domain which, in turn, imply high computational costs. In addition, one should mention that the dynamic of dispersive equations is dramatically changed depending they are set on the whole space or in a periodic domain: in the latter case, small amplitude waves cannot scatter to infinity and stay in the computational domain forever.
Instead, one can imagine a more appropriate strategy based on the transparent boundary conditions (TBC): this consists in deriving suitable boundary conditions so that the solution calculated in the computational domain is an approximation of the exact solution restricted to the computational domain. These artificial boundary conditions are called absorbing boundary conditions (ABC) if they lead to a well-posed initial boundary value problem where some energy is absorbed at the boundaries. See \cite{AABES} for a review on the techniques used to construct such transparent or artificial boundary conditions for the Schr\"odinger equation.

In this paper, we focus on the {\it  linearized KdV-BBM equation}
\begin{equation}\label{kdvLI}
\displaystyle
\partial_t (u-\alpha\partial_{xx}u)+c\partial_x u+\varepsilon\partial_{xxx}u=0,\quad \forall t>0,\quad \forall x\in\mathbb{R},
\end{equation}
where $\alpha, \varepsilon$ are dispersion parameter and $c$ is a velocity.  The computation of continuous and discrete transparent boundary conditions for the pure (BBM) case ($\varepsilon=0$) was recently performed in \cite{BMN}. In the pure (KdV) case 
($\alpha=0$),  continuous transparent boundary conditions were derived in \cite{Zheng06,ZWH}. Recently, exact transparent boundary conditions both continuous and discrete were derived and implemented in \cite{BEL-V}. The discrete boundary conditions were derived for a upwind (first order) and a centered (second order) spatial discretization. The time discretization is based on the Crank-Nicolson scheme. The discrete artificial boundary conditions (DTBC) were previously introduced in \cite{AES03,AESS12,Eh2001,Eh2008,EA01} mainly for the time dependent Schr\"odinger equation. These (DTBC) are superior since they are by construction perfectly adapted to the used interior scheme and thus retain the stability properties of the underlying discretization method and theoretically do not produce any reflections when compared to the discrete whole space solution. However, in the case of the linearized (KdV) equation, the boundary conditions are not explicit and a numerical inverse $\mathcal{Z}$-transformation is needed  which produces a numerical error and create instabilities for large time simulations (see \cite{AES03,Zi02003}).\\

The aim of this paper is to propose an alternative procedure to carry out numerically the computation of the inverse $\mathcal{Z}$-transformation through a stable method and use it to obtain discrete transparent boundary conditions with no restriction on the simulation time. We shall also explore approximate {\it explicit} boundary conditions by expanding the exact discrete boundary conditions in various asymptotic regime. 

The paper is organized as follows. In section \ref{sec2}, we first recall the derivation of continuous transparent boundary conditions for the linearized (KdV-BBM) equation and show a stability result. In section \ref{sec3}, we focus on discrete transparent boundary conditions: we show a consistency result and establish sufficient stability conditions which in turn guarantees convergence of our numerical procedure.  In section \ref{sec4}, we carry out numerical tests: we consider test cases with Gaussian and wave packet initial data. For large time simulations, we also derive approximate {\it explicit} discrete boundary conditions and show numerically stability of these conditions.

\section{\label{sec2} Transparent boundary conditions for the linear KdV-BBM equation}

In this section, we recall the derivation of the exact artificial boundary conditions. To do so, we consider the initial boundary value problem 
{\setlength\arraycolsep{1pt}
\begin{eqnarray}
\label{kdv}
\displaystyle
\partial_t (u-\alpha\partial_{xx}u)+c\partial_x u+\varepsilon\,\partial_{xxx}u=0,\quad\forall t>0,\quad\forall x\in\mathbb{R},\\
\label{ci}
\displaystyle
u(0,x)=u_0(x),\quad\forall x\in\mathbb{R},\\
\label{bc-inf}
\displaystyle
\lim_{x\to\infty}u(t,x)=\lim_{x\to-\infty}u(t,x)=0,
\end{eqnarray}%
}
\noindent where  $u_0$ is compactly supported in a finite computational interval $[x_\ell,\,x_r]$ with $x_\ell<x_r$. The constants  $c\in\mathbb{R}$ and $\alpha, \varepsilon>0$ are respectively a velocity and two dispersion parameters. The construction of (continuous) artificial boundary conditions associated to problem (\ref{kdv}-\ref{bc-inf}) is established by considering the problem on the complementary of $[x_{\ell},\,x_r]$
{\setlength\arraycolsep{1pt}
\begin{eqnarray}
\label{kdv-t}
\displaystyle
\partial_t (u-\alpha\partial_{xx}u)+c\,\partial_x u+\varepsilon\partial_{xxx}u=0,\quad\forall t>0,\quad\forall x<x_{\ell},\quad\forall x>x_r,\\
\displaystyle
u(0,x)=0,\quad \forall x<x_{\ell},\quad\forall x>x_r,\\
\displaystyle
\lim_{x\to\infty}u(t,x)=\lim_{x\to-\infty}u(t,x)=0.
\end{eqnarray}
} 

\subsection{Exact boundary conditions}

In order to derive transparent boundary conditions, we write (\ref{kdv-t}) as a first order system with respect to the $x$ variable:
\begin{equation}
  \label{eq:kdvsyst}
  \partial_x
  \begin{pmatrix}
    u\\v\\w
  \end{pmatrix}
=
\begin{pmatrix}
  0 & 1 & 0 \\
  0 & 0 & 1 \\
  \displaystyle
  -{\varepsilon}^{-1} \partial_t & \displaystyle -{\varepsilon}^{-1}\,c & \alpha\varepsilon^{-1}\partial_t
\end{pmatrix}
\begin{pmatrix}
  u\\v\\w
\end{pmatrix}.
\end{equation}

\noindent
Next, the problem being homogeneous in time, we use Laplace transform so that (\ref{eq:kdvsyst}) is transformed into a classical first order differential system with parameter $s\in\mathbb{C}$ with $\Re(s)>0$:
\begin{equation}
  \label{eq:kdvsyst:l}
  \partial_x
  \begin{pmatrix}
    \hat u\\\hat v\\\hat w
  \end{pmatrix}
=
\begin{pmatrix}
  0 & 1 & 0 \\
  0 & 0 & 1 \\
  \displaystyle
  -{\varepsilon}^{-1}\,s & \displaystyle -{\varepsilon}^{-1}\,c & \alpha\varepsilon^{-1}s
\end{pmatrix}
\begin{pmatrix}
  \hat u\\\hat v\\\hat w
\end{pmatrix}:=\mathcal{A}_{\alpha, \varepsilon}(s,c)\begin{pmatrix}
  \hat u\\\hat v\\\hat w
\end{pmatrix}.
\end{equation}

\noindent
The general solutions of this system of ODE are given explicitly by
\begin{equation}\label{sol:edo}
\displaystyle
\left(\begin{array}{c} \hat u\\ \hat v \\ \hat w\end{array}\right)= e^{\lambda_1(s)\,x}\,\mathcal{V}_1(s)+e^{\lambda_2(s)\,x}\,\mathcal{V}_2(s)+e^{\lambda_3(s)\,x}\,\mathcal{V}_3(s),\quad x<x_{\ell},\quad x>x_r,
\end{equation}
\noindent
where $\lambda_k(s), k=1,2,3$ are the roots of $P(s,c,\alpha,\varepsilon,\lambda)=s+c\lambda-\alpha s\lambda^2+\varepsilon\lambda^3=0$ and $\mathcal{V}_k=\left(1, \lambda_k, \lambda_k^2\right)^T$ are the right eigenvectors of the matrix $\mathcal{A}_{\alpha,\varepsilon}(s,c)$ associated to eigenvalue $\lambda_k$. 
Let $j$ be $j=e^{2i\pi/3}$. The roots $\lambda_k$ are given by
\begin{equation}\label{roots}
\displaystyle
\lambda_k(s)=\frac{\alpha\,s}{3\varepsilon}+j^{k-1}\zeta(s)^{1/3}-\left(\frac{c}{3\varepsilon}-\frac{\alpha^2s^2}{9\varepsilon^2}\right)\frac{1}{j^{k-1}\zeta(s)^{1/3}},\quad k=1,2,3,
\end{equation}
\noindent
with
$$
\displaystyle
\zeta(s)=\frac{1}{2}\left(-q-\sqrt{q^2+\frac{4}{27}p^3}\right),\quad p=\frac{c}{\varepsilon}-3(\frac{\alpha s}{3\varepsilon})^2,\quad q=\frac{s}{\varepsilon}+\frac{\alpha s c}{3\varepsilon^2}-2\left(\frac{\alpha s}{3\varepsilon}\right)^3.
$$

\noindent
\begin{proposition}
For all $\varepsilon>0$ and for all $\alpha\geq 0$, the roots $\lambda_k(s), k=1,2,3$ possess the following separation property:
$$
\displaystyle
\Re(\lambda_1(s))<0,\quad \Re(\lambda_2(s))>0,\quad \Re(\lambda_3(s))>0,\quad\forall \Re(s)>0.
$$
\end{proposition}

\begin{proof}
Without loss of generality, one can assume $\varepsilon=1$. Let us first show that no roots crosses the imaginary axis when $\Re(s)>0$. We assume that there exists $\lambda=i\xi\in\; i\R$, such that
$$
\displaystyle
s+ci\xi-\alpha s(i\xi^2)+(i\xi)^3=0.
$$
\noindent
Then, one finds that $s\in i\mathbb{R}$. Now let us assume that $c$ and $s$ are fixed and $\alpha\in[0, M]$ for some $M>0$. Then, by applying the Rouch\'e's theorem, one finds that the number of roots  of $P(s,c,\alpha,.)$ with a positive real part is given by
$$
\displaystyle
N^+(s,c,\alpha)=\frac{1}{2i\pi}\oint_{\mathcal{C}_R}\frac{P'(s,c,\alpha,\lambda)}{P(s,c,\alpha,\lambda)}d\lambda
$$
\noindent
with $\mathcal{C}_R=\left\{z\in\mathbb{C},\:|\:\Re(z)\geq 0,\quad |z|\leq R\right\}$ with $R>0$ sufficiently large depending on $s,c$ and $M$. The function 
$N^+(c,s,.):[0, M]\to\mathbb{N}$ is continuous so that one can deduce that $N^+(c,s,\alpha)=N^+(c,s,0)$ for all $\alpha\in[0, M]$. It is proved in \cite{BEL-V} that $N^+(c,s,0)=2$ which concludes the proof of the proposition.

\end{proof}

\noindent
Now, we search for solutions $(\hat u, \hat v, \hat w)^T$ such that $\lim_{x\to\infty}\hat u(s,x)=0$. It is satisfied provided that we impose the condition
\begin{equation}\label{bcr}
\displaystyle
\mathcal{V}_1(s)\wedge\left(\begin{array}{c}\hat u(s,x_r)\\ \hat v(s,x_r)\\ \hat w(s,x_r)\end{array}\right)=0,
\end{equation}
\noindent
which in turn provides the following two boundary conditions
\begin{equation}\label{bcr-s}
\displaystyle
\partial_x \hat u(s,x_r)=\lambda_1(s)\hat u(s,x_r),\qquad \partial_{xx}\hat u(s,x_r)=\lambda_1^2(s)\hat u(s,x_r).
\end{equation}

\noindent
A similar argument to obtain solutions $(\hat u, \hat v, \hat w)^T$ such that $\lim_{x\to-\infty}\hat u(s,x)=0$. We therefore have to impose the condition
\begin{equation}\label{bcl}
\displaystyle
\mathcal{V}_2(s)\wedge\mathcal{V}_3(s)\cdot \left(\begin{array}{c}\hat u(s,x_\ell)\\ \hat v(s,x_\ell)\\ \hat w(x,x_\ell)\end{array}\right)=0,
\end{equation}
which gives the following boundary condition.
$$
\displaystyle
\partial_{xx}\hat u(s,x_{\ell})-(\lambda_2(s)+\lambda_3(s))\partial_x\hat u(s,x_{\ell})+\lambda_2\lambda_3\hat u(s,x_{\ell})=0.
$$
\noindent
By using relations between roots $\lambda_i$, one finds that this condition is equivalent to
\begin{equation}\label{bcl-s}
\displaystyle
\partial_{xx}\hat u(s,x_{\ell})+\left(\lambda_1(s)-\frac{\alpha\,s}{\varepsilon}\right)\partial_x\hat u(s,x_{\ell})+\left(\lambda_1(s)^2-\frac{\alpha\,s}{\varepsilon}\lambda_1(s)+\frac{c}{\varepsilon}\right)\hat u(s,x_{\ell})=0.
\end{equation}

\noindent
Written in time variables, the boundary conditions (\ref{bcr-s}) and (\ref{bcl-s}) read
\begin{equation}\label{bcrl-t}
\begin{array}{ll}
\displaystyle
\partial_x u(t,x_r)=\mathcal{L}^{-1}(\lambda_1(s))\ast u(t,x_r),\qquad \partial_{xx}u(t,x_r)=\mathcal{L}^{-1}(\lambda_1^2(s))\ast u(t,x_r),\\
\displaystyle
\partial_{xx} u(t,x_{\ell})+\mathcal{L}^{-1}(\lambda_1(s)-\frac{\alpha\,s}{\varepsilon})\ast\partial_x u(t,x_{\ell})+\mathcal{L}^{-1}(\lambda_1(s)^2-\frac{\alpha\,s}{\varepsilon}\lambda_1(s))\ast u(t,x_{\ell})+\frac{c}{\varepsilon}u(t,x_{\ell})=0.
\end{array}
\end{equation}
\noindent
A natural question is whether such truncation procedure provides absorbing boundary conditions. We prove the following $H^1$ stability result.

\begin{proposition}\label{condstab}
Assume that 
$$
\displaystyle
\frac{c}{2}+\varepsilon\left(\Re(\lambda_1^2(i\xi))-\frac{|\lambda_1(i\xi)|^2}{2}\right)-\alpha\Re(i\xi\lambda_1(i\xi))\geq 0,\quad\forall \xi\in\mathbb{R}.
$$

Then the  problem 
\begin{equation}
  \label{full-cont-syst}
\left \{  \begin{array}{ll}
\displaystyle    \partial_t (u-\alpha\partial_{xx}u)+c\partial_x u+\varepsilon\,\partial_{xxx}u=0,& (t,x) \in \mathbb{R}_*^+ \times (x_\ell,x_r),\\
\displaystyle u(0,x)=u_0(x),& x \in (x_\ell,x_r),\\
\displaystyle
\partial_x u(t,x_r)=\mathcal{L}^{-1}(\lambda_1(s))\ast u(t,x_r),\quad \partial_{xx}u(t,x_r)=\mathcal{L}^{-1}(\lambda_1^2(s))\ast u(t,x_r),& t \in \mathbb{R}_*^+,\\
\displaystyle
\partial_{xx} u(t,x_{\ell})+\mathcal{L}^{-1}(\lambda_1(s)-\frac{\alpha\,s}{\varepsilon})\ast\partial_x u(t,x_{\ell})+& t \in \mathbb{R}_*^+,\\
\displaystyle \qquad \qquad \qquad \qquad \qquad \mathcal{L}^{-1}(\lambda_1(s)^2-\frac{\alpha\,s}{\varepsilon}\lambda_1(s))\ast u(t,x_{\ell})+\frac{c}{\varepsilon}u(t,x_{\ell})=0, & 
  \end{array}
\right .
\end{equation}
is $H^1$-stable. More precisely for any $t>0$, the generalized kinetic energy satisfies
\[
\int_{x_\ell}^{x_r} u^2(t,x)+\alpha(\partial_x u)^2(t,x)\, dx \le \int_{x_\ell}^{x_r} u_0^2(x)+\alpha(\partial_x u_{0})^2 \,dx .
\]
\end{proposition}

\begin{remark}
The root $\lambda_1(s)$ is defined for all $s\in\mathbb{C}$ such that $\Re(s)>0$. We define $\lambda_1(i\xi)$ with $\xi\in\mathbb{R}$ as 
$$
\displaystyle
\lambda_1(i\xi)=\lim_{\eta\to 0^+}\lambda_1(\eta+i\xi).
$$
\end{remark}

\proof

\noindent
Let us first compute the time derivative of $\|u(t,\cdot)\|^2_{L^2(x_\ell,x_r)}$:
{\setlength\arraycolsep{1pt}
\begin{eqnarray*}
\dsp \frac{d}{dt}  \int_{x_{\ell}}^{x_r} u^2(t,x) \, dx & = & \dsp - \int_{x_{\ell}}^{x_r} \d_x \left (c u^2 +2 \varepsilon u \d_x^2 u - \varepsilon (\d_x u)^2-2\alpha u\d^2_{tx}u \right )(t,x) \, dx ,\\
\dsp
& = & \dsp \left (c u^2 +2 \varepsilon u \d_x^2 u - \varepsilon (\d_x u)^2-2\alpha u\d^2_{x,t}u \right )(t,x_{\ell}) \\
\dsp
&&-\left (c u^2 +2 \varepsilon u \d_x^2 u - \varepsilon (\d_x u)^2-2\alpha u\d^2_{xt}u \right )(t,x_r).
\end{eqnarray*}}
Then the generalized kinetic energy is
{\setlength\arraycolsep{1pt}
\begin{eqnarray*}
\dsp \int_{x_\ell}^{x_r} u^2(t,x)+\alpha(\partial_x u)^2(t,x) \, dx & = &\int_{x_\ell}^{x_r} u_0^2(x)+\alpha(\partial_x u_{0})^2 \,dx \\
&& \dsp + I_t^1 \left (c u^2 +2 \varepsilon u \d_x^2 u - \varepsilon (\d_x u)^2-2\alpha u\d^2_{xt}u  \right )(\cdot,x_\ell)\\ &&- I_t^1\left (c u^2 +2 \varepsilon u \d_x^2 u - \varepsilon (\d_x u)^2-2\alpha u\d^2_{xt}u \right )(\cdot ,x_r) ,\\
& :=&\int_{x_\ell}^{x_r} u_0^2(x)+\alpha(\partial_x u_{0})^2 \,dx +  J_\ell - J_r.
\end{eqnarray*}
}
The problem is $H^1$ stable if $J_\ell\leq 0$ and $J_r\geq 0$. Let us fix $T>0$ and set  $U = u(t,x_\ell) \1_{[0,T]}$ and  $V = \d_x u(t,x_\ell) \1_{[0,T]}$. One has
{\setlength\arraycolsep{1pt}
\begin{eqnarray*}
\displaystyle
J_\ell&=&\int_0^{\infty}cU^2-\varepsilon V^2+2\varepsilon U\left(Op(\frac{\alpha s}{\varepsilon}-\lambda_1(s))V+Op(\frac{\alpha s}{\varepsilon}\lambda_1(s)-\frac{c}{\varepsilon}-\lambda_1(s)^2)U\right)-2\alpha UV'dt,\\
\displaystyle
&=&\frac{1}{2\pi}\int_{\mathbb{R}}c|\hat U|^2-\varepsilon|\hat V|^2+2\varepsilon\overline{\hat U}\left((\frac{i\alpha\xi}{\varepsilon}-\lambda_1(i\xi))\hat V+(\frac{i\alpha\xi}{\varepsilon}\lambda_1(i\xi)-\frac{c}{\varepsilon}-\lambda_1^2(i\xi))\hat U\right)-2i\alpha\xi\overline{\hat U}\hat V\,d\xi,\\
\displaystyle
&=&\frac{1}{2\pi}\int_{\mathbb{R}}(2i\alpha\xi\lambda_1(i\xi)-c-2\varepsilon\lambda_1^2(i\xi))|\hat U|^2-\varepsilon|\hat V|^2-2\varepsilon\lambda_1(i\xi)\overline{\hat U}\hat V,\\
&\leq&\frac{1}{2\pi}\int_{\mathbb{R}} (2\Re(i\alpha\xi\lambda_1(i\xi))-c-2\varepsilon\Re(\lambda_1^2(i\xi))|\hat U|^2+\varepsilon |\lambda_1(i\xi)|^2|\hat U|^2d\xi\leq 0,
\end{eqnarray*}}
\noindent by assumption on the sign of $\displaystyle\frac{c}{2}+\varepsilon\left(\Re(\lambda_1^2(i\xi))-\frac{|\lambda_1(i\xi)|^2}{2}\right)-\alpha\Re(i\xi\lambda_1(i\xi))$. Now, let us set $U = u(t,x_r) \1_{[0,T]}$. We have
{\setlength\arraycolsep{1pt}
\begin{eqnarray*}
\displaystyle
J_r&=&\int_0^\infty cU^2-\varepsilon (Op(\lambda_1)U)^2+2\varepsilon U\,Op(\lambda_1^2)U-2\alpha U\,Op(\lambda_1)U',\\
\displaystyle
&=&\frac{1}{2\pi}\int_{\mathbb{R}}\left( c-\varepsilon|\lambda_1(i\xi)|^2+2 \varepsilon\Re(\lambda_1^2(i\xi))-2\Re(i\alpha\xi\lambda_1(i\xi))\right)|\hat U|^2 \geq 0. 
\end{eqnarray*}
}
\noindent
This completes the proof of the proposition.
\endproof

\begin{proposition}
The stability condition given in Prop. \ref{condstab} is always fulfilled:
\[
\forall \xi \in \R, \ \ \frac{c}{2} + \varepsilon \left ( \Re (\lambda_1^2(i\xi)) - \frac{|\lambda_1(i\xi)|^2}{2} \right ) - \alpha \Re ( i \xi \lambda_1 (i \xi)) \ge 0.
\]
\end{proposition}
\proof
We let $\lambda_1(i \xi) = a+ ib$. The stability condition writes:
\[
I = \frac{c}{2} + \varepsilon \left ( a^2 -b^2 - \frac{a^2 +b^2 }{2} \right )  + \alpha \xi b = \frac{c}{2} + \varepsilon \frac{a^2 - 3b^2}{2} + \alpha \xi b.
\]
The roots $\lambda_k(s)$, $k=1,2,3$,  with $\Re(s)>0$ fulfill
\[
s + c \lambda  - \alpha s  \lambda^2 + \varepsilon \lambda^3 =0.
\]
Writing $s= \eta + i \xi$ with $\xi \in \R$ and $\eta>0$ and taking the limit as $\eta \to 0$ we get
\[
i \xi + c \lambda(i \xi) - \alpha i \xi \lambda^2(i \xi) + \varepsilon \lambda^3(i \xi) = 0.
\]
By taking the real part of this equation we obtain:
\[
0= \varepsilon \left ( a^3 - 3 a b^2 \right )  +2 \alpha a \xi b + c a = 2 a I. 
\]
Either $I=0$  and the stability condition is fulfilled, either $\Re (\lambda_1(i \xi))=0$.\\

In order to study the latter case, we perform an asymptotic expansion in the expression of $\lambda_k(\eta+ i \xi)$ given in \eqref{roots} as $\eta \to 0$. The asymptotic expansions of various terms involved in the definition of $\lambda_k$ are given by
\[
\begin{array}{rcl}
p & = &  \dsp \left ( \frac{c}{\varepsilon} + 3 \left (\frac{\alpha \xi}{3 \varepsilon} \right)^2 \right ) - 6 i \xi \left ( \frac{\alpha}{3 \varepsilon} \right )^2 \eta + O(\eta^2),\\& := &p_0 + i p_1 \eta + O(\eta^2), \\[2mm]
q & = &  \dsp \left (\frac{i\xi}{\varepsilon} + \frac{\alpha c i \xi}{3 \varepsilon^2} +2 i \xi^3 \left (\frac{\alpha}{3 \varepsilon} \right )^3 \right ) + \eta \dsp \left (\frac{1}{\varepsilon} + \frac{\alpha c}{3 \varepsilon^2} + 6 \xi^2 \left ( \frac{ \alpha}{3 \varepsilon} \right )^3  \right ) + O(\eta^2), \\
& := & i q_0 + q_1 \eta + O(\eta^2), \\[2mm]
\dsp q^2 + \frac{4}{27}p^3 & = & \dsp \left ( - q_0^2 + \frac{4}{27} p_0^3 +2 i \eta q_0 q_1 + i \frac{4}{27} 3 p_0^2 p_1 \eta + O(\eta^2) \right ), \\&:= &A_0 + i \eta A_1 + O(\eta^2),\\[2mm]
\zeta &= & \dsp  - \frac{1}{2} \left ( i q_0 + A_0^{1/2} + \eta \left ( q_1 + i \frac{A_1}{2 A_0^{1/2}} \right ) +O(\eta^2) \right ), \\ &:= &\zeta_0 + \eta \zeta_1 + O(\eta^2).
\end{array}
\]
The terms $p_j$, $q_j$ and $A_j$ with $j=0,1$ are real.\\
{\bf Case $A_0 < 0$~:} $A_0 = - a_0^2$ and $A_0^{1/2} = i a_0$. We let $E=\dsp  -\left ( \frac{q_0+a_0}{2} \right )^{1/3}$. Then
\[
\begin{array}{lcl}
\lambda_1(\eta + i \xi) & = &\dsp \frac{\alpha s}{3 \varepsilon} + \zeta^{1/3} - \frac{p}{3} \zeta^{-1/3}, \\
& = & \dsp \frac{\alpha i \xi}{3 \varepsilon}+  e^{i\pi/6} E - \frac{p_0}{3E}e^{-i\pi/6} + \eta \left (\frac{\alpha}{3 \varepsilon} -  \frac{\zeta_1}{3 E^2}e^{2i\pi/3}- i \frac{p_1}{3 E}e^{-i\pi/6} - \frac{p_0 \zeta_1}{9 E^4}e^{i\pi/3} \right ) + O(\eta^2),
\end{array}
\]
and we have
\[
\Re(\lambda_1(i\xi)) = \frac{\sqrt{3}}{2} \left (E - \frac{p_0}{3E} \right ) , \ \ \Im(\lambda_1(i\xi)) = \frac{\alpha \xi}{3 \varepsilon} + \frac{1}{2}\left (E+ \frac{p_0}{3E} \right ).
\]
We are now in the case where $\dsp \Re(\lambda_1(i\xi))= \frac{\sqrt{3}}{2} \left (E - \frac{p_0}{3 E} \right ) =0$ which implies that $\dsp E^2 = \frac{p_0}{3}$ which also writes
\[
a_0 (a_0 + q_0) =0.
\]
Since $a_0>0$, we have $\dsp \sqrt{q_0^2-\frac{4}{27} p_0^3} + q_0=0$, which leads to $\dsp 0=p_0= \frac{c}{\varepsilon} + 3 \frac{ \alpha^2 \xi^2}{9 \varepsilon^2}$. If $c>0$ there is no $\xi \in \R$ such that $\Re(\lambda_1(i\xi)) =0$. If $c \le 0$, $\Re(\lambda_1(i\xi)) =0$ if and only if $\dsp \xi = \pm \sqrt{-\frac{3 c \varepsilon}{\alpha^2}}$. In this cases we have $E=0$ and
\[
\begin{array}{lcl}
I & = & \dsp \frac{c}{2} - \frac{3 \varepsilon (\Im(\lambda_1(i\xi)))^2}{2} + \alpha \xi \Im(\lambda_1(i\xi), \\
& = & \dsp \frac{c}{2} - \frac{3 \varepsilon }{2} \left ( \frac{\alpha \xi}{3 \varepsilon} +E \right )^2 + \alpha \xi \left ( \frac{\alpha \xi}{3 \varepsilon} + E \right ) = 0.
\end{array}
\]
{\bf Case $A_0 \ge 0$~:} $A_0 = a_0^2$ and $A_0^{1/2} = a_0$. We let $E=\dsp  \left ( \frac{a_0^2+q_0^2}{4} \right )^{1/6}$. Then $\frac{1}{2} \left ( A_0^{1/2} + iq_0 \right ) = E^3 e^{3i \theta}$ with $\theta \in [-\pi/6,\pi/6]$ and
\[
\begin{array}{lcl}
\lambda_1(\eta + i \xi) & = &\dsp \frac{\alpha s}{3 \varepsilon} + \zeta^{1/3} - \frac{p}{3} \zeta^{-1/3}, \\
& = & \dsp \frac{\alpha i \xi}{3 \varepsilon} - E e^{i \theta} + \frac{p_0}{3E}e^{-i \theta} + \eta \left (\frac{\alpha}{3 \varepsilon} +  \frac{\zeta_1}{3 E^2}e^{-2 i \theta} + i \frac{p_1}{3 E}e^{-i \theta} + \frac{p_0 \zeta_1}{9 E^4}e^{- 4i \theta} \right ) + O(\eta^2).
\end{array}
\]
The real and imaginary parts are given by
\[
\Re(\lambda_1(i\xi)) = - E \cos \theta + \frac{p_0}{3E} \cos \theta , \ \ \Im(\lambda_1(i\xi)) = \frac{\alpha \xi}{3 \varepsilon} - \left ( E + \frac{p_0}{3E} \right )\sin \theta .
\]
We are now in the case where $\dsp \Re(\lambda_1(i\xi))= - \left ( E - \frac{p_0}{3E} \right ) \cos \theta =0$ which implies that $\dsp E^2 = \frac{p_0}{3}$.
In this case we have
\[
\begin{array}{lcl}
I & = & \dsp \frac{c}{2} - \frac{3 \varepsilon (\Im(\lambda_1(i\xi)))^2}{2} + \alpha \xi \Im(\lambda_1(i\xi), \\
& = & \dsp \frac{c}{2} - \frac{3 \varepsilon }{2} \left ( \frac{\alpha \xi}{3 \varepsilon} - 2 E \sin \theta \right )^2 + \alpha \xi \left ( \frac{\alpha \xi}{3 \varepsilon} -2 E \sin \theta \right ), \\
& = & \dsp \frac{c}{2} + \frac{3 \varepsilon}{2} \left ( \frac{\alpha^2 \xi^2}{9 \varepsilon^2} - 4 \frac{p_0}{3} \sin^2 \theta  \right ), \\
& \ge & \dsp  \frac{\varepsilon}{2} \left (  \frac{c}{\varepsilon} + 3 \left (\frac{\alpha \xi}{3 \varepsilon}\right )^2 - p_0  \right )  \ \ \text{ since } \theta \in [- \pi/6, \pi/6] \text{ and } p_0 \ge 0,\\
& \ge & 0.
\end{array}
\]
\endproof

\section{\label{sec3} Discrete transparent boundary conditions}

It is not possible to compute explicitly the inverse Laplace transform of $\lambda_k$, $k=1,2,3$ and thus we cannot obtain a closed form of the boundary conditions. It is therefore difficult to discretize the transparent boundary conditions \eqref{bcrl-t} without any other knowledge. In \cite{BEL-V} and \cite{BMN}, the construction of the discrete transparent boundary conditions for the approximation of the linearized Korteweg-de Vries equation (lKdV) ($\alpha=0$) and the linearized Benjamin-Bona-Mahoney equation (lBBM) ($\varepsilon=0$) is made on fully discrete numerical schemes. In the case of the (lBBM) case, the space differential operator is of order two and it is possible to give explicit formulas both for the continuous and discrete transparent boundary conditions. It is not the case when one deals with (lKdV) where the space differential operator is of order 3. In the continuous case, the explicit inverse Laplace transform is not available. This issue is also met at the discrete level where a numerical procedure is used to compute numerically the inverse
$\mathcal{Z}$ transform. However, it requires an implementation with quadruple
precision floating number in order to avoid instabilities as time
becomes large (see \cite{BEL-V} for more details). Here we propose an
alternative approach to invert numerically the $\mathcal{Z}-$transform which allows to construct ``explicit'' coefficient of discrete kernels.


\subsection{Design and computation of discrete transparent boundary conditions}

In this section, we derive discrete transparent boundary conditions associated to the centered-Crank Nicolson discretization of the linear KdV-BBM equation:
{\setlength\arraycolsep{1pt}
\begin{eqnarray}
\displaystyle
{u^{n+1}_j-u^n_j}&-&\lambda_B\left({u^{n+1}_{j+1}-2u_j^{n+1}+u_{j-1}^{n+1}}-{u^{n}_{j+1}+2u_j^{n}-u_{j-1}^{n}}\right)\nonumber\\
\displaystyle
&+&\frac{\lambda_H}{4}\left(u_{j+1}^{n+1}-u_{j-1}^{n+1}\right)+\frac{\lambda_H}{4}\left(u_{j+1}^{n}-u_{j-1}^{n}\right)\nonumber\\
\displaystyle
&+&\frac{\lambda_D}{4}\left(u_{j+2}^{n+1}-2u_{j+1}^{n+1}+2\,u_{j-1}^{n+1}-u_{j-2}^{n+1}\right)\nonumber\\
\label{c-cn}
\displaystyle
&+&\frac{\lambda_D}{4}\left(u_{j+2}^{n}-2u_{j+1}^{n}+2\,u_{j-1}^{n}-u_{j-2}^{n}\right)=0,\:\:\forall j=0,\dots,J,
\end{eqnarray}
}
\noindent
with 
$$
\displaystyle
\lambda_H=\frac{c\delta t}{\delta x},\quad \lambda_D=\frac{\varepsilon\delta t}{\delta x^3},\quad \lambda_B=\frac{\alpha}{\delta x^2}.
$$
Here, $\delta t>0$ denotes the time step and $\delta x>0$ the space step. We set $J = (x_r-x_\ell)/{\delta x}$. The approximation of the exact solution $u(t,x)$ at points $j\delta x$ and instants $n\delta t$ with $0\leq j\leq J$ and $ n\in\mathbb{N}$ is
 $u_j^n \approx u(n\delta t,x_\ell+j\delta x)$.

In order to derive appropriate artificial boundary conditions, we follow the same procedure
as in Section \ref{sec2}, but on a purely discrete level. First we apply the $\mathcal{Z}$-transform
with respect to the time index $n$, which is the discrete analogue of the Laplace transform
in time, to the partial difference equation (\ref{c-cn}).  The standard
definition reads
$$
\displaystyle
\hat u(z) =\mathcal{Z}\{(u^n)_n\}(z)=\sum_{k=0}^{\infty}u^k\,z^{-k},\quad |z|>R>0,
$$
\noindent
where $R$ is the convergence radius of the Laurent series and $z\in\mathbb{C}$. Denoting $\hat u_j=\hat u_j(z)$
the $\mathcal{Z}-$transform of the sequence $(u_j^{(n)})_{n\in\mathbb{N}}$, we obtain from (\ref{c-cn}) the homogeneous {\it fourth order difference equation}
{\setlength\arraycolsep{1pt}
\begin{eqnarray}
\displaystyle
\hat{u}_{j+2}-\left(2-\frac{\lambda_H}{\lambda_D}+\frac{4\lambda_B}{\lambda_D}\frac{z-1}{z+1}\right)\hat{u}_{j+1}&+&\left(\frac{4}{\lambda_D}+\frac{8\lambda_B}{\lambda_D}\right)\frac{z-1}{z+1}\hat{u}_j\nonumber\\
\label{Z-c-cn}
\displaystyle
&+&\left(2-\frac{\lambda_H}{\lambda_D}-\frac{4\lambda_B}{\lambda_D}\frac{z-1}{z+1}\right)\hat{u}_{j-1}-\hat u_{j-2}=0.
\end{eqnarray}
}
\noindent
The characteristic polynomial associated to this linear recurrence relation is given by
\begin{equation}\label{char-Z-cn}
\displaystyle
P(r)=r^4-(2-a+\mu p(z))r^3+\left(\frac{4a}{\lambda_H}+2\mu\right)p(z)r^2+(2-a-\mu p(z))r-1=0.
\end{equation}
\noindent
with 
$$
\displaystyle
a=\frac{\lambda_H}{\lambda_D},\quad \mu=\frac{4\lambda_B}{\lambda_D}, \quad p(z)=\frac{z-1}{z+1}=\frac{1-z^{-1}}{1+z^{-1}}.
$$
\noindent
We prove the following separation properties on the roots $r_k(z), k=1,2,3,4$:

\begin{proposition}
Assume $\varepsilon>0, \alpha\geq 0$, $\delta x,\delta t>0$ and $c\in\mathbb{R}$. Then,  the roots of $P$ are well separated according to 
$$
\displaystyle
|r_1(z)|<1,\quad |r_2(z)|<1,\qquad |r_3(z)|>1,\quad |r_4(z)|>1
$$
which defines the discrete separation properties. As a consequence, there is a smooth parameterization of the ``stable'' (respectively ``unstable'') subspace $\mathbb{E}^s(z)$ (resp $\mathbb{E}^u(z)$) of solutions to (\ref{Z-c-cn}) which decrease to $0$ as $j\to+\infty$ (respectively $j\to-\infty$) for $|z|>R$ with $R$ large enough. 
\end{proposition}

\begin{proof}
First let us note that $p:z\mapsto\frac{z-1}{z+1}$ maps $\{z\in\mathbb{C}\,|\,|z|>1\}$ onto $\{z\in\mathbb{C}\,|\,\Re(z)>0\}$. Now let us assume that there exists $z$ such that $|z|>1$ and there is a root $r=e^{i\theta}$ of $P$.
Then, one finds:
$$
\displaystyle
4i\sin(\theta)\left (\cos(\theta)-\frac{2-a}{2}\right )=\left(2\mu(\cos(\theta)-1)-\frac{4a}{\lambda_H}\right)p(z).
$$
\noindent
This equation holds for some $z$ such that $|z|>1$ only if
$$
\displaystyle
\sin(\theta)\left(\cos(\theta)-\frac{2-a}{2}\right )=\left(2\mu(\cos(\theta)-1)-\frac{4a}{\lambda_H}\right)=0.
$$
\noindent
If $\sin(\theta)=0$ or $\cos(\theta)=1-a/2$ then $\left(2\mu(\cos(\theta)-1)-\frac{4a}{\lambda_H}\right)\leq -4a/\lambda_H<0$ so there is no roots of $P$ on the unit circle for all $z$ such that $|z|>1$.  Let us order the four roots of $P$ as $|r_1(z)|\leq |r_2(z)|\leq |r_3(z)|\leq |r_4(z)|$. Since $|r_1(z)r_2(z)r_3(z)r_4(z)|=1$, one has $|r_1(z)|<1<|r_4(z)|$. There remains to locate $r_2(z)$ and $r_3(z)$. Let $p(z)\to\infty$: one has 
$r_4(z)\sim \mu p(z)$ whereas there are $r_i(z), i=1,2,3$ remains bounded and converge to the solutions of 
$$
\displaystyle
\mu r^3-(\frac{4a}{\lambda_H}+2\mu)r^2+\mu r=0.
$$
\noindent
One deduce then that $r_1(z)\to 0$ as $p(z)\to+\infty$ and a straightforward computation shows that $r_1(z)\sim-\frac{1}{\mu p(z)}$. Finally $r_2(z),r_3(z)$ converge to the roots of
$$
\displaystyle
r^2-(2+\frac{4a}{\mu\lambda_H})r+1=0
$$
\noindent
We easily deduce that $|r_2(z)|<1<|r_3(z)|$. This concludes the proof of the proposition.
\end{proof}
\noindent
According to this proposition, we set 
$$
\displaystyle
S^s(z)=r_1(z)+r_2(z),\: P^s(z)=r_1(z)r_2(z),\quad S^u(z)=r_3(z)+r_4(z),\:P^u(z)=r_3(z)r_4(z)
$$
and the characteristic polynomial $P$ admits the factorization
$$
\displaystyle
P(r)=\left(r^2-S^u(z)r+P^u(z)\right)\left(r^2-S^s(z)r+P^s(z)\right)
$$
\noindent
The discrete transparent boundary conditions are written as follows. On the left boundary, one must have
$$
\displaystyle
\left(\hat u_{-2}, \hat u_{-1}, \hat u_0, \hat u_1\right)\in \mathbb{E}^u(z)
$$
\noindent
which is also equivalent to the following boundary conditions
\begin{equation}\label{left-bc}
\displaystyle
\hat u_1-S^u(z)\,\hat u_0\,+P^u(z)\,u_{-1}=0,\quad \hat u_0-S^u(z)\,\hat u_{-1}\,+P^u(z)\,u_{-2}=0.
\end{equation}
\noindent
On the other hand, one must have on the right boundary
$$
\displaystyle
\left(\hat u_{J-1}, \hat u_{J}, \hat u_{J+1}, \hat u_{J+2}\right)\in \mathbb{E}^s(z)
$$
\noindent
which is also written as
\begin{equation}\label{right-bc}
\displaystyle
\hat u_{J+2}-S^s(z)\,\hat u_{J+1}\,+P^s(z)\,\hat u_{J}=0,\quad \hat u_{J+1}-S^s(z)\,\hat u_{J}\,+P^s(z)\,\hat u_{J-1}=0.
\end{equation}
\noindent
The coefficients of $P$ admits a singularity at $z=-1$ which in turn implies bad behavior of the coefficients in the expansion of $S^u, P^u, S^s, P^s$. In order to remove this singularity, we will consider alternative boundary conditions by multiplying (\ref{right-bc}) and (\ref{left-bc}) by $1+z^{-1}$. Inverting the $\mathcal{Z}-$transform, one finds that the left and right boundary conditions are written as:
\begin{equation}\label{left-bc-n}
\displaystyle
u_1^{n+1}+u_1^n+\tilde s^u\ast_d\,u_0^{n+1}+\tilde p^u\ast_d\,u_{-1}^{n+1}=0,\quad u_0^{n+1}+u_0^n+\tilde s^u\ast_d\,u_{-1}^{n+1}+\tilde p^u\ast_d\,u_{-2}^{n+1}=0,
\end{equation}
\begin{equation}\label{right-bc-n}
\displaystyle
u_{J+2}^{n+1}+u_{J+2}^n+\tilde s^s\ast_d\,u_{J+1}^{n+1}+\tilde p^s\ast_d\,u_{J}^{n+1}=0,\quad u_{J+1}^{n+1}+u_{J+1}^n+\tilde s^u\ast_d\,u_{J}^{n+1}+\tilde p^u\ast_d\,u_{J-1}^{n+1}=0,
\end{equation}
where the sequences $\tilde s^u,\tilde p^u$ and $\tilde s^s, \tilde p^s$ are defined as
$$
\begin{array}{ll}
\displaystyle
\tilde S^s(z)=(1+z^{-1})S^s(z)=\sum_{n=0}^\infty\frac{\tilde s^s_n}{z^n},\quad \tilde P^s(z)=(1+z^{-1})P^s(z)=\sum_{n=0}^\infty\frac{\tilde p^s_n}{z^n},\\
\displaystyle
\tilde S^u(z)=(1+z^{-1})S^u(z)=\sum_{n=0}^\infty\frac{\tilde s^u_n}{z^n},\quad \tilde P^u(z)=(1+z^{-1})P^u(z)=\sum_{n=0}^\infty\frac{\tilde p^u_n}{z^n}.
\end{array}
$$
\noindent
The relations \eqref{left-bc-n} and \eqref{right-bc-n} allow to compute the ghost values $u_{-2}$, $u_{-1}$, $u_{J+1}$ and $u_{J+2}$.
Now there remains to compute the coefficients $\tilde s_n^u,\tilde p_n^u,\tilde s^s_n, \tilde p^s_n$ for all $n\in\mathbb{N}$. For the Schr\"odinger equation \cite{AES03,Zi02003} and for the (lKdV)  equation ($\alpha=0$) \cite{BEL-V}, the numerical procedure to compute the coefficients is as follows: if one set $U(z)=\displaystyle\sum_{k=0}^\infty u_k\,z^{-k}$ for all $|z|>R$, the coefficients $u_k$ are recovered by the formula
$$
\displaystyle
u_n=\frac{r^n}{2\pi}\int_0^{2\pi}U(r\,e^{i\phi})e^{in\phi}d\phi,\quad \forall n\in\mathbb{N},
$$
\noindent
for some $r>R$ and the approximation of these integrals are done by using the Fast Fourier Transform.
In \cite{AES03,Zi02003,BEL-V}, one has $R=1$ so that the numerical procedure is unstable as $n\to \infty$ due to truncation errors.  Here, we propose an alternative approach based on the use of the relation between coefficients and roots. We set $x=z^{-1}$ and one has to solve the following system
\begin{equation}\label{sys-coeff}
\begin{array}{llll}
\displaystyle
S^s(x)+S^u(x)=2-a+\mu\frac{1-x}{1+x},\\
\displaystyle
P^u(x)+P^s(x)+S^u(x)S^s(x)=\left(\frac{4a}{\lambda_H}+2\mu\right)\frac{1-x}{1+x},\\
\displaystyle
P^u(x)S^s(x)+P^s(x)S^u(x)=-\left(2-a-\mu\frac{1-x}{1+x}\right),\\
\displaystyle
P^u (x)P^s(x)=-1.\\
\end{array}
\end{equation}

\noindent
As mentioned previously, we compute instead an expansion of $\tilde S^u, \tilde S^s, \tilde P^u, \tilde P^s$: the system satisfied by these quantities is given by
\begin{equation}\label{sys-coeff-tilde}
\begin{array}{llll}
\displaystyle
\tilde S^s(x)+\tilde S^u(x)=(2-a)(1+x)+\mu(1-x),\\
\displaystyle
(1+x)\tilde P^u(x)+(1+x)\tilde P^s(x)+\tilde S^u(x)\tilde S^s(x)=\left(\frac{4a}{\lambda_H}+2\mu\right)(1-x^2),\\
\displaystyle
\tilde P^u(x)\tilde S^s(x)+\tilde P^s(x)\tilde S^u(x)=-\left((2-a)(1+x)^2-\mu(1-x^2)\right),\\
\displaystyle
\tilde P^u (x)\tilde P^s(x)=-(1+x)^2.\\
\end{array}
\end{equation}

Now, the coefficients satisfy the linear recurrence relations for all $n\geq 1$
\begin{equation}\label{coeff-expl}
\begin{array}{llll}
\displaystyle
\tilde s^s_n+\tilde s^u_n=\sigma_1^n,\\
\displaystyle
\tilde p^u_n+\tilde p^s_n+\tilde s^s_0\tilde s^u_n+\tilde s^s_n\tilde s^u_0=\sigma_2^n-\tilde p^u_{n-1}-\tilde p^s_{n-1}-\sum_{k=1}^{n-1}\tilde s^s_k\tilde s^u_{n-k},\\
\displaystyle
\tilde s^s_0\tilde p^u_n+\tilde s^s_n\tilde p^u_0+\tilde p^s_0\tilde s^u_n+\tilde p^s_n\tilde s^u_0=\sigma_3^n-\sum_{k=1}^{n-1}\tilde s^s_k\tilde p^u_{n-k}-\sum_{k=1}^{n-1}\tilde p^s_k\tilde s^u_{n-k},\\
\displaystyle
\tilde s^s_0\tilde p^u_n+\tilde s^s_n\tilde p^u_0=\sigma^4_n-\sum_{k=1}^{n-1}\tilde p^s_k\tilde p^u_{n-k}.
\end{array}
\end{equation}
\noindent
whereas the coefficients for $n=0$ satisfy the non linear system:
\begin{equation}\label{sys0}
\displaystyle
\tilde s^s_0+\tilde s^u_0=\sigma_1^0,\quad\tilde p^s_0+\tilde p^u_0+\tilde s^s_0\tilde s^u_0=\sigma_2^0,\:\:
\tilde p^u_0\tilde s^s_0+\tilde p^s_0\tilde s^u_0=\sigma_3^0,\quad \tilde p^s_0\tilde p^u_0=\sigma_4^0.
\end{equation}
\noindent
Here the sequences $\sigma_k, k=1,2,3,4$ are given by
$$
\begin{array}{ll}
\displaystyle
\sigma_1=(2-a+\mu)\delta_0+(2-a-\mu)\delta _1,\quad \sigma_2=\left(\frac{4a}{\lambda_H}+2\mu\right)(\delta_0-\delta_2),\\
\displaystyle
\sigma_3=-(2-a-\mu)\delta_0-2(2-a)\delta _1-(2-a+\mu)\delta_2, \sigma_4=-\delta_0-2\delta_1-\delta_2.
\end{array}
$$
\noindent
The nonlinear system is solved numerically simply by computing the roots of $P$ for $z^{-1}=x=0$ and the recurrence relation (\ref{coeff-expl}) is implemented directly. Note that a $4\times 4$ matrix has to be inverted. The invertibility is ensured by the separation of the roots at $x=0$. We plot the coefficients $\tilde s^s_n$ on left curves of Figure \ref{fig:coeff} and we see that they decrease as $n^{-3/2}$ just as in the BBM case or for the Schrodinger equation \cite{Eh2001}.

\begin{figure}[htbp]
\begin{center}
\begin{tabular}{cc}
\includegraphics[width=0.48\textwidth]{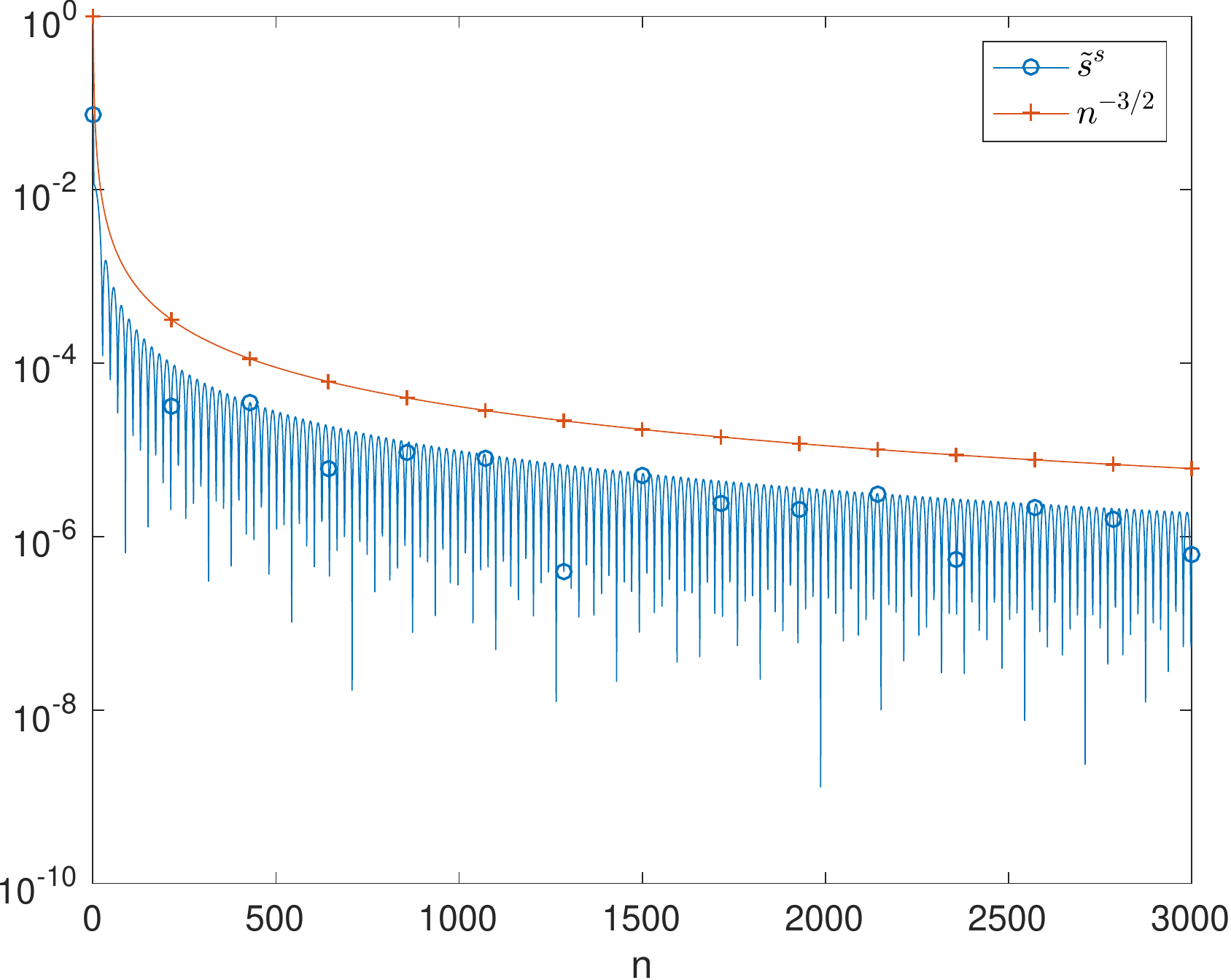} & \includegraphics[width=0.48\textwidth]{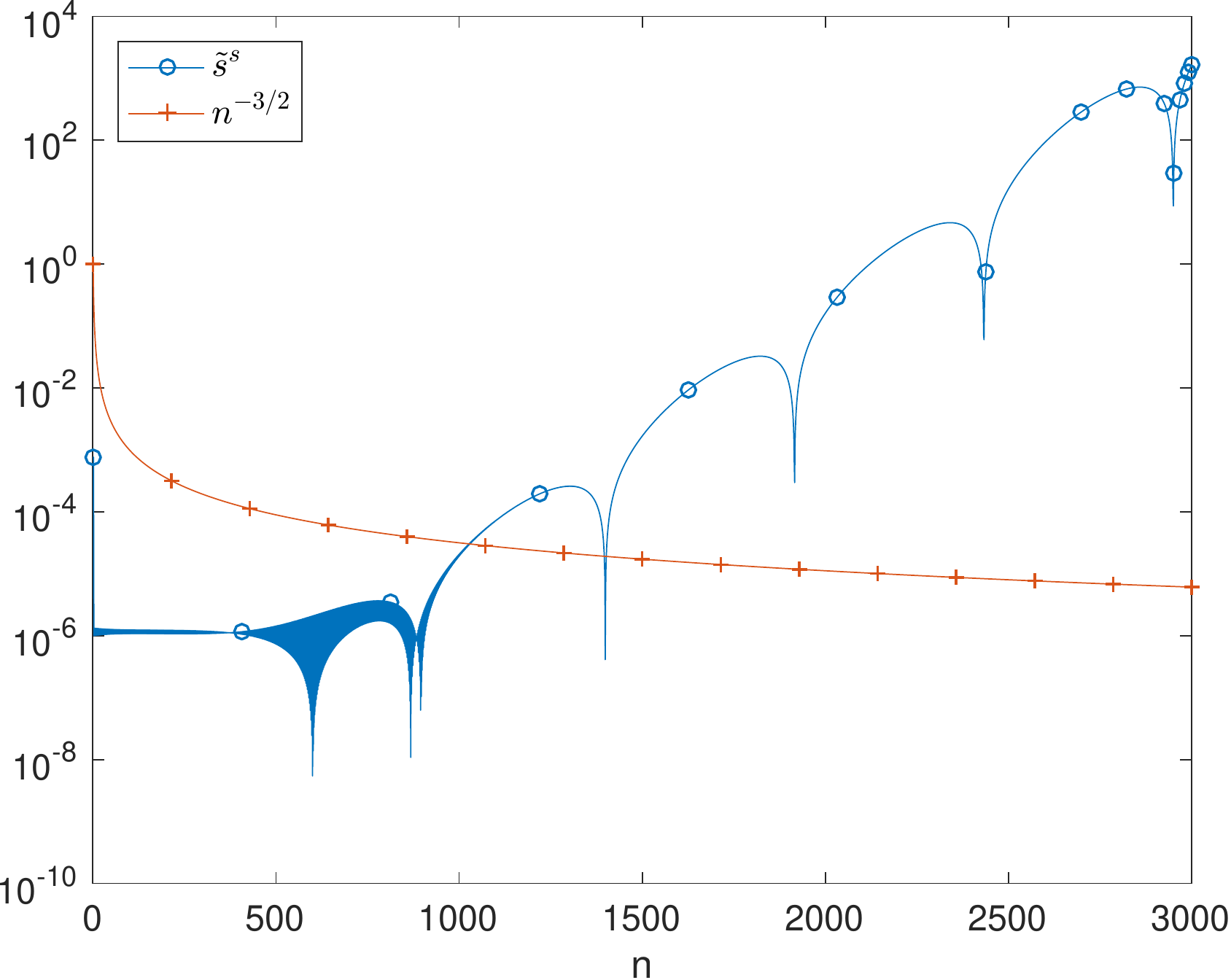}\\
$\delta t=10^{-4}$ & $\delta t=10^{-2}$
\end{tabular}
\end{center}  
  \caption{Coefficients $\tilde s^s_n$  with $\delta x=2^{-18}$, $\alpha=\delta=1$ and $c=2$}
  \label{fig:coeff}
\end{figure}


Though in the limit $\delta x\to 0$, the roots of $P$ are not separated which implies that the linear systems that are to be solved in our procedure are not well conditioned which increases the impact of numerical errors: see right curves of figure \ref{fig:coeff}. 


\noindent
Note that this ``bad'' behavior is observed for spatial steps $\delta x$ much smaller than those used in \cite{BEL-V}. In order to deal with this particular problem for small spatial steps $\delta x$, we will carry out an asymptotic expansion of the coefficients $\tilde{s}^s_n, \tilde{s}^u_n$ and $\tilde{p}^s_n, \tilde{p}^u_n$ in the limit $\delta x\to 0$ and truncate these expansions to a given order $p$ larger than $2$ in order to preserve the accuracy of the scheme. 

In the limit $\delta t\to 0$, the divergence of coefficient is indeed stronger since it is easily proved that
$$
\displaystyle
r_1(z=+\infty)\sim_{\delta t\to 0}-\frac{\varepsilon\delta t}{4\alpha\delta x},\qquad r_4(z=+\infty)\sim_{\delta t\to 0}\frac{4\alpha\delta x}{\varepsilon \delta t}.
$$

\subsection{Consistency and stability of discrete transparent boundary conditions}

In this section, we check the consistency of the discrete transparent boundary conditions (\ref{left-bc}) and (\ref{right-bc}) with the continuous boundary conditions.  As an application, we shall obtain  an asymptotic expansion of the convolution coefficients $\tilde{s}^s_n, \tilde{s}^u_n$ and $\tilde{p}^s_n, \tilde{p}^u_n$. First, let us prove the following proposition

\begin{proposition}
Let $s\in\mathbb{C}$  such that $\Re(s)>0$.  Set $z=\exp(s\delta t)$ and $\mu(z)=\frac{2(z-1)}{\delta t(z+1)}$ and assume that the roots $\lambda_i(s)$ are distinct. Then the roots $r_1,r_2,r_3,r_4$ of the characteristic equation (\ref{char-Z-cn}) admits a smooth expansion with respect to $\delta t,\delta x\to 0$ and expands as: 
$$
\begin{array}{lll}
\displaystyle
r_1=1+\delta x\tilde \lambda_1(s,\delta x)=1+\delta x \lambda_1(s)+\delta x^2\lambda_1^1+\delta x\,O(\delta x^2+s^2\delta t^2),\\[1mm]
\displaystyle
r_2=-1+\frac{\alpha s}{\varepsilon}\delta x-\frac{\alpha^2s^2}{2\varepsilon^2}\delta x^2+\delta x\,O(\delta x ^2+s^2\delta t^2),\\[2mm]
\displaystyle
r_3=1+\delta x\tilde\lambda_2(s,\delta x)=1+\delta x \lambda_2(s)+O(\delta x^2+s^2\delta t^2),\\[1mm]
\displaystyle
r_4=1+\delta x\tilde\lambda_3(s,\delta x)=1+\delta x\lambda_3(s)+O(\delta x^2+s^2\delta t^2).
\end{array}
$$
\noindent
with 
$$
\displaystyle
\lambda_1^1=\frac{\lambda_1(s)\left(\alpha s\lambda_1(s)^2-2c\lambda_1(s)+3s\right)}{3\lambda_1^2(s)-2\alpha s\lambda_1(s)+c}.
$$
\end{proposition}

\begin{proof}
As a first step, we compute an expansion of the three roots that bifurcate from $1$.
Let us rewrite (\ref{char-Z-cn}) as
$$
\displaystyle
\varepsilon(r-1)^3\frac{r+1}{2}-\alpha\delta x\,\mu(z)r(r-1)^2+c\delta x^2r(\frac{r+1}{2})(r-1)+\delta x^3\,\mu(z)\,r^2=0.
$$
\noindent
Now, let us set $r=1+\delta x \lambda$: one finds
$$
\displaystyle
\varepsilon(1+\frac{\lambda\delta x}{2})\lambda^3-\alpha p(z)(1+\lambda\delta x)\lambda^2+c(1+\lambda\delta x)(1+\frac{\lambda\delta x}{2})\lambda+(1+\lambda\delta x)^2p(z)=0.
$$
\noindent
Letting $\delta t,\delta x\to 0$, one obtains $\mu(z)=s(1+O(s^2\delta t^2))$ and
\begin{equation}\label{char-cont}
\displaystyle
\varepsilon\lambda^3-\alpha s\lambda^2+c\lambda+s=0. 
\end{equation}
Recall that we have chosen $\lambda_i(s)$ roots of (\ref{char-cont}) such that $\Re(\lambda_1(s))<0$, $\Re(\lambda_2(s))>0$ and $\Re(\lambda_3(s))>0$ and assumed that the roots are distinct so that we can apply the implicit function theorem: there are three roots $r_1,r_2,r_4$ which bifurcates from $1$ and expands as
$$
\displaystyle
r_1=1+\delta x \lambda_1(s)+O(\delta x^2+s^2\delta t^2),\quad r_k=1+\delta x\lambda_{k-1}(s)+O(\delta x^2+s^2\delta t^2),\quad k=3,4.
$$
Next, we compute an expansion of the eigenvalue bifurcating from $-1$: the implicit function theorem do apply and we obtain
$$
r_2=-1+\frac{\alpha s}{\varepsilon}\delta x+O(s^2\delta t^2+\delta x^2).
$$
\noindent
This concludes the proof of the proposition.
\end{proof}

Now we can check the consistency of the discrete transparent boundary conditions with the continuous ones. To simplify the presentation, we assume that $[x_\ell,x_r]=[0,1]$.
\begin{proposition}\label{prop-dl}
Let $u$ be a smooth solution of the (KdV-BBM) system \eqref{full-cont-syst}. For all $x\in[-2\delta x, 1+2\delta x]$, we define the $\mathcal{Z}$-transform of $(u(n\delta t,x))_{n\in\mathbb{N}}$ by 
$$
\displaystyle
\hat u(z,x)=\sum_{n=0}^{\infty}\frac{u(n\delta t,x)}{z^n}.
$$
\noindent
Then, for all $K\subset\mathbb{C}^+$ and all $s\in{K}$, one has for the left boundary conditions:
$$
\begin{array}{ll}
\displaystyle
\hat u(e^{s\delta t}, \delta x)-S^u(e^{s\delta t})\hat u(e^{s\delta t},0)+P^u(e^{s\delta t})\hat u(e^{s\delta t},-\delta x)={\delta x^2}\,O(\delta t+\delta x),\\
\displaystyle
\hat u(e^{s\delta t}, 0)-S^u(e^{s\delta t})\hat u(e^{s\delta t},-\delta x)+P^u(e^{s\delta t})\hat u(e^{s\delta t},-2\delta x)={\delta x^2}\,O(\delta t+\delta x),
\end{array}
$$
\noindent
whereas on the right hand side
$$
\begin{array}{ll}
\displaystyle
\hat u(e^{s\delta t}, 1+2\delta x)-S^s(e^{s\delta t})\hat u(e^{s\delta t},1+\delta x)+P^s(e^{s\delta t})\hat u(e^{s\delta t},1)={\delta x}\,O(\delta t+\delta x),\\
\displaystyle
\hat u(e^{s\delta t}, 1+\delta x)-S^s(e^{s\delta t})\hat u(e^{s\delta t},1)+P^s(e^{s\delta t})\hat u(e^{s\delta t},1-\delta x)={\delta x}\,O(\delta t+\delta x).
\end{array}
$$
\end{proposition}

\begin{proof}
Let us first check left boundary conditions. First, from proposition \ref{prop-dl}, one finds that
$$
\begin{array}{ll}
\displaystyle
S^u(e^{s\delta t})=2+\delta x\left(\tilde \lambda_2(s,\delta x)+\tilde \lambda_3(s,\delta x)\right),\\
\displaystyle
P^u(e^{s\delta t})=1+\delta x\left(\tilde\lambda_2(s,\delta x)+\tilde \lambda_3(s,\delta x)\right)+\delta x^2\tilde\lambda_2(s,\delta x)\tilde\lambda_3(s,\delta x).
\end{array}
$$
\noindent
By inserting these expansions in the discrete transparent boundary conditions, one finds
{\setlength\arraycolsep{1pt}
\begin{eqnarray}
\displaystyle
\hat u(e^{s\delta t}, \delta x)&-&S^u(e^{s\delta t})\hat u(e^{s\delta t},0)+P^u(e^{s\delta t})\hat u(e^{s\delta t},-\delta x)=\hat u(e^{s\delta t}, \delta x)-2\hat u(e^{s\delta t},0)+\hat u(e^{s\delta t},-\delta x)\nonumber\\
\displaystyle
&&-\delta x\left(\tilde \lambda_2(s,\delta x)+\tilde \lambda_3(s,\delta x)\right)\left(\hat u(e^{s\delta t}, 0)-\hat u(e^{s\delta t}, -\delta x)\right)\nonumber\\
\displaystyle
&&+\delta x^2\tilde \lambda_2(s,\delta x)\tilde \lambda_3(s,\delta x)\hat u(e^{s\delta t},-\delta x),\nonumber\\
\displaystyle
&=&\delta x^2\left(\frac{\partial^2}{\partial x^2}\hat u(e^{s\delta t},0)-(\lambda_2(s)+\lambda_3(s))\frac{\partial}{\partial x}\hat u(e^{s\delta t},0)+\lambda_2(s)\lambda_3(s)\hat u(e^{s\delta t},0)+O(\delta x)\right),\nonumber\\
\displaystyle
&=&\delta x^2\left(\delta t^{-1}\left(\frac{\partial^2}{\partial x^2}\mathcal{L}u(s,0)-(\lambda_2(s)+\lambda_3(s))\frac{\partial}{\partial x}\mathcal{L}u(s,0)+\lambda_2(s)\lambda_3(s)\mathcal{L}u(s,0)\right)+O(\delta t+\delta x)\right),\nonumber\\
\displaystyle
&=&\delta x^2\,O\left(\delta x+\delta t\right).\nonumber
\end{eqnarray}
}
\noindent
The proof of consistency of the second boundary condition on the left is carried out similarly. Let us now consider boundary conditions on the right. Note that $S^s$ and $P^s$ are written as
{\setlength\arraycolsep{1pt}
\begin{eqnarray}
\displaystyle
\hat u(e^{s\delta t}, 1+\delta x)&-&S^s(e^{s\delta t})\hat u(e^{s\delta t},1)+P^s(e^{s\delta t})\hat u(e^{s\delta t},1-\delta x)\nonumber\\
\displaystyle
&=&\hat u(e^{s\delta t}, 1+\delta x)-\hat u(e^{s\delta t}, 1-\delta x)-\delta x\lambda_1(s)\left(\hat u(e^{s\delta t},1)+\hat u(e^{s\delta t}, 1-\delta x)\right)\nonumber\\
&&-\frac{\alpha s}{\varepsilon}\left(\hat u(e^{s\delta t},1)-\hat u(e^{s\delta t}, 1-\delta x)\right)+O(\delta x^2),\nonumber\\
\displaystyle
&=&2\delta x\left(\partial_x \hat u(e^{s\delta t},1)-\lambda_1(s)\hat u(e^{s\delta t},1)+O(\delta x)\right)=\delta x\,O(\delta t+\delta x).\nonumber
\end{eqnarray}
}
\noindent
The proof of consistency of the second boundary condition is similar. This completes the proof of the proposition.
\end{proof}

\begin{remark}
Note that the order of accuracy is one order lower on the right hand side. This is due to the additional mode that bifurcates from $-1$ which is a pure numerical artifact. 
\end{remark}

\noindent
Let us now write a stability result for discrete transparent boundary conditions. For that purpose, we introduce
 $\mathcal{A}^s(z)$ and $\mathcal{A}^u(z)$ the Hermitian matrices
$$
\displaystyle
\mathcal{A}^s=\left(\begin{array}{cc} \alpha^s(z) & \gamma^s(z)\\[2mm]
                                                           \overline{\gamma^s(z)} & \beta^s(z)
                                                           \end{array}\right),\quad 
 \mathcal{A}^u=\left(\begin{array}{cc} \alpha^u(z) & \gamma^u(z)\\[2mm]
                                                           \overline{\gamma^u(z)} & \beta^u(z)
                                                           \end{array}\right)                                                          
                                                           $$
with
{\setlength\arraycolsep{1pt}
\begin{eqnarray*}
\alpha^s(z)&=&\frac{|z+1|^2}{2}\Re(-p^s(z)),\\
\displaystyle
\beta^s(z)&=&\frac{|z+1|^2}{2}\left(\Re(s^s(z)^2-p^s(z)+(a-2)s^s(z)\right)-\mu\frac{z-\bar z}{2\,i}\Im(s^s(z)),\\
\displaystyle
\gamma^s(z)&=&\frac{|z+1|^2}{4}\left(\overline{s^s(z)}-s^s(z)p^s(z)-(a-2)p^s(z)\right)-\mu\frac{z-\bar z}{2i}\frac{p^s(z)}{2\,i}
\end{eqnarray*}
}
and
{\setlength\arraycolsep{1pt}
\begin{eqnarray*}
\displaystyle
\alpha^u(z)&=&\frac{|z+1|^2}{2}\Re(p^u(z)),\\
\displaystyle
\beta^u(z)&=&\frac{|z+1|^2}{2}\left(\Re(p^u(z)-s^u(z)^2-(a-2)s^u(z)\right)-\mu\frac{z-\bar z}{2\, i}\Im(s^u (z))\\
\displaystyle
\gamma^u(z)&=&\frac{|z+1|^2}{4}\left(p^u(z)s^u(z)-\overline{s^u(z)}+(a-2)p^u(z)\right)+\mu\frac{z-\bar z}{2\,i}\frac{p^u(z)}{2\,i}.
\end{eqnarray*}
}

\begin{proposition}
Let $u_j^n$ with $j\in[-1,\:J+1]$ and $n\in\mathbb{N}$ solution of  (\ref{c-cn}) with the discrete transparent boundary conditions (\ref{left-bc-n}) and (\ref{right-bc-n}). Denote $\mathcal{E}_n$
\begin{equation}
  \label{eq:disc_energy}
\displaystyle
\mathcal{E}_n=\sum_{j=1}^J\frac{(u_j^n)^2}{2}+\alpha\sum_{j=0}^J\frac{(u_{j+1}^n-u_j^n)^2}{2\delta x^2}. 
\end{equation}
Assume that for all $\theta\in[-\pi,\,\pi]$ the Hermitian matrices $\mathcal{A}^s(e^{i\theta})$ and $\mathcal{A}^u(e^{i\theta})$ are positive semi-definite.  Then the boundary conditions (\ref{left-bc-n}) and (\ref{right-bc-n}) are dissipative:  
$$
\displaystyle
\forall N\in\mathbb{N},\qquad \mathcal{E}_N-\mathcal{E}_0=-\mathcal{R}_{\ell}-\mathcal{R}_r\leq 0
$$ 
with
$$
\begin{array}{ll}
\displaystyle
\mathcal{R}_r=\frac{\lambda_D}{8\pi}\int_{-\pi}^{\pi}\langle \left(\begin{array}{c}\widehat{u_{J-1}}(e^{i\theta})\\\widehat{{u}_{J}}(e^{i\theta})\end{array}\right);\mathcal{A}^s(e^{i\theta})\left(\begin{array}{c}\widehat{u_{J-1}}(e^{i\theta})\\\widehat{u_J}(e^{i\theta})\end{array}\right)\rangle\,d\theta\geq 0,\\
\displaystyle
\mathcal{R}_\ell=\frac{\lambda_D}{8\pi}\int_{-\pi}^{\pi}\langle \left(\begin{array}{c}\widehat{u_{-1}}(e^{i\theta})\\\widehat{{u}_{0}}(e^{i\theta})\end{array}\right);\mathcal{A}^u(e^{i\theta})\left(\begin{array}{c}\widehat{u_{-1}}(e^{i\theta})\\\widehat{u_0}(e^{i\theta})\end{array}\right)\rangle\,d\theta\geq 0.
\end{array}
$$
\end{proposition}

\begin{remark}
In the pure BBM case, the discrete transparent boundary conditions are proved to be dissipative and for all $n\geq 0$, one has $\mathcal{E}_n\leq \mathcal{E}_0$: see \cite{BMN} for a proof. Note that in this later case, the discrete transparent boundary conditions, only one ghost point is added at the end of each boundary and dissipativity is proved only by checking the sign of a function defined on the unit circle at each end of the domain. Here, we see that we have to check that two Hermitian matrices are positive semi definite. The size of the matrices is determined by the number of ghost points added at each boundary. These conditions are hardly verified in the general case and we will show later that the boundary conditions are indeed dissipative through direct numerical simulations. A generak framework to study the dissipativity of the transparent numerical boundary conditions for evolution equations can be found in \cite{coulombel:hal-01369975}.
\end{remark}

\begin{proof}
Multiply equation (\ref{c-cn}) by $\displaystyle v_j^n=\frac{u_j^n+u_j^{n+1}}{2}$ 
 and sum over all $j=0,\cdots,J$: one finds
 {\setlength\arraycolsep{1pt}
 \begin{eqnarray}
 \mathcal{E}_{n+1}&-&\mathcal{E}_n-\lambda_Bv_{J+1}^n\left((u_{J+1}^{n+1}-u_J^{n+1})-(u_{J+1}^n-u_J^n)\right)+\lambda_B\,v_{-1}^n\left((u_{0}^{n+1}-u_{-1}^{n+1})-(u_{0}^n-u_{-1}^n)\right)\nonumber\\
 &+&\frac{\lambda_H}{2}(v_{J+1}^nv_J^n-v_0^nv_{-1}^n)+\frac{\lambda_D}{2}\left(v_{J+2}^n\,v_J^n+v_{J+1}^nv_{J-1}^n-v_1^nv_{-1}^n-v_0^nv_{-2}^n\right)-\lambda_D\left(v_{J+1}^nv_J^n-v_0^nv_{-1}^n\right)=0.
 \end{eqnarray}
 }
 
 \noindent
 Denote $r_\ell^n$ the contribution of boundary terms at the left end of the domain and $r_r^n$ at the right end. Then, one has
 $$
 \begin{array}{ll}
 \displaystyle
  r_\ell^n=-\frac{\lambda_D}{4}\left(2(v_{1}^n\,v_{-1}^n+v_{0}^nv_{-2}^n)+2(a-2)v_{0}^nv_{-2}^n-\mu v_{-1}^n\left((u_{0}^{n+1}-u_{-1}^{n+1})-(u_{0}^n-u_{-1}^n)\right) \right),\\
 \displaystyle
 r_r^n=\frac{\lambda_D}{4}\left(2(v_{J+2}^n\,v_J^n+v_{J+1}^nv_{J-1}^n)+2(a-2)v_{J+1}^nv_J^n-\mu v_{J+1}^n\left((u_{J+1}^{n+1}-u_J^{n+1})-(u_{J+1}^n-u_J^n)\right) \right).
 \end{array}
 $$
\noindent
We sum the equations for all $n=0,\dots,N-1$: one finds
$$
\displaystyle
\mathcal{E}_N-\mathcal{E}_0+R_r^N+R_\ell^N=0, \quad R_r^N=\sum_{n=0}^{N-1}r_r^n,\quad R_\ell^N=\sum_{n=0}^{N-1}r_\ell^n.
$$
Let us deal with the right hand side. By applying Plancherel's theorem for $\mathcal{Z}$-transform, one finds
\begin{equation}\label{rn}
\displaystyle
R_r^N=\frac{\lambda_D}{8\pi}\int_{-\pi}^{\pi}\frac{|z+1|^2}{2}\left(\widehat{u_{J+2}}\overline{\widehat{u_{J}}}+\widehat{u_{J+1}}\overline{\widehat{u_{J-1}}}+(a-2)\widehat{u_{J+1}}\overline{\widehat{u_{J}}}\right)(e^{i\theta})-\mu\frac{z-\bar z}{2}\overline{\widehat{u}_{J+1}}(\widehat{u_{J+1}}-\widehat{u_J})(e^{i\theta})d\theta.
\end{equation}
\noindent
Recall that the discrete transparent boundary conditions are given by
$$
\displaystyle
\widehat{u_{J+2}}(z)=s^s(z)\widehat{u_{J+1}}(z)-p^s(z)\widehat{u_J}(z),\quad \widehat{u_{J+1}}(z)=s^s(z)\widehat{u_{J}}(z)-p^s(z)\widehat{u_{J-1}}(z).
$$
 By substituting these relations into (\ref{rn}), one finds $R_r^N=\mathcal{R}_r$. Similarly, one finds $R_\ell^N=\mathcal{R}_\ell$. This concludes the proof of the proposition.
\end{proof}

\section{\label{sec4} Numerical Results}
We propose in this section to illustrate the behaviour of the numerical solutions  to \eqref{kdvLI} when we use the numerical scheme \eqref{c-cn} complemented with the boundary conditions \eqref{left-bc-n} and \eqref{right-bc-n}. 

\subsection{Computation of reference solutions \label{sub:ref}}
In order to plot convergence curves, we need to compare to reference solutions. We use two techniques to compute reference solutions to \eqref{kdvLI}. The first technique is dedicated to the linear Korteweg-de Vries equation
\begin{equation}
  \label{lin-kdv}
\partial_t u + \varepsilon \partial_{xxx}u=0.  
\end{equation}
The fundamental solution to \eqref{lin-kdv} is
$$
E(t,x)=\frac{1}{\sqrt[3]{3\varepsilon t}} \Ai\left (\frac{x}{\sqrt[3]{3 \varepsilon t}} \right ),
$$
where $\Ai(\cdot)$ is the Airy function. Then the exact solution to \eqref{lin-kdv} is given by
$$
u_{\textrm{ref}}(t,x)=E(t,x) * u_0(x),
$$
where $*$ denotes the convolution product on the whole real axis.

Such a fundamental solution is not known for the general equation \eqref{kdvLI}. In this case, we apply a Fourier transform to \eqref{kdvLI} and obtain
$$
(1+\alpha \xi^2)\partial_t\hat{u}(t,\xi)+i(c\xi-\varepsilon \xi^3)\hat{u}(t,\xi)=0,
$$
where $\xi$ stands for the Fourier variable. The reference solution is therefore obtained by 
$$
u_{\mathrm{ref}}(t,x)=\mathscr{F}^{-1}\left (\exp\left(i \frac{\varepsilon \xi^3-c\xi}{1+\alpha \xi^2}t\right)\hat{u_0}(\xi)  \right ).
$$
The computation is made with fast Fourier transforms and periodic boundary conditions. The extent of the computational domain is chosen large enough to avoid any spurious effects of the boundary conditions.

\subsection{Case 1: exact transparent boundary conditions}

The numerical scheme is given by \eqref{c-cn} coupled with the boundary conditions \eqref{left-bc-n} and \eqref{right-bc-n}. Then, the linear system we have to solve is given by
\begin{equation}
  \label{eq:linsys}
  A\mathbf{u}^{n+1}=B\mathbf{u}^n + \mathbf{s}^n
\end{equation}
where $A,B\in M_{J+5,J+5}(\mathbb{R})$ and $\mathbf{u}^n,\mathbf{u}^{n+1},\mathbf{s}^n\in \mathbb{R}^{J+5}$ with $\mathbf{u}^n_j=u_j^n$,
$$
A=
\begin{pmatrix}
  \tilde{p}_0^u & -\tilde{s}_0^u & 1              & 0    &      &      &        &   \\
  0             & \tilde{p}_0^u  & -\tilde{s}_0^u & 1    & 0    &      &        &   \\
 -1             & c_-            & c_0            & c_+  & 1    & 0    &        &   \\
 0              & -1             & c_-            & c_0  & c_+  & 1    & 0      &   \\
                &  \ddots        & \ddots         &\ddots&\ddots&\ddots& \ddots &   \\
                &                &    0           & -1   & c_-  & c_0  & c_+     & 1 \\
                &                &                &  0   & \tilde{p}^s_0 & -\tilde{s}^s_0  & 1     & 0 \\
                &                &                &      & 0   & \tilde{p}^s_0 & -\tilde{s}^s_0    & 1 
\end{pmatrix}
, \quad
B=
\begin{pmatrix}
  0             & 0              & -1             & 0    &      &      &        &   \\
  0             & 0              & 0              & -1   & 0    &      &        &   \\
  1             & c_+            & c_0            & c_-  & -1   & 0    &        &   \\
 0              &  1             & c_+            & c_0  & c_-  & -1   & 0      &   \\
                &  \ddots        & \ddots         &\ddots&\ddots&\ddots& \ddots &   \\
                &                &    0           &  1   & c_+  & c_0  & c_-     & -1 \\
                &                &                &  0   & 0  & 0  & -1     & 0 \\
                &                &                &      & 0   & 0 & 0    & -1 
\end{pmatrix}
$$
and
$$
\mathbf{s}^n=
\begin{pmatrix}
  \sum_{k=0}^n \tilde{s}^u_{n+1-k}u_0^k-\tilde{p}^u_{n+1-k}u_{-1}^k\\
  \sum_{k=0}^n \tilde{s}^u_{n+1-k}u_1^k-\tilde{p}^u_{n+1-k}u_{0}^k\\
0\\
\vdots\\
0\\
  \sum_{k=0}^n \tilde{s}^s_{n+1-k}u_J^k-\tilde{p}^s_{n+1-k}u_{J-1}^k\\
  \sum_{k=0}^n \tilde{s}^s_{n+1-k}u_{J+1}^k-\tilde{p}^s_{n+1-k}u_{J}^k
\end{pmatrix}.
$$
The constants $c_-$, $c_0$ and $c_+$ take the values
$$
c_-=2-a-\mu, \qquad c_0=\frac{4a}{\lambda_H}+2\mu, \qquad c_+=a-2-\mu.
$$
The computational domain is $(t,x)\in [0,4]\times[0,1]$. The evolution of the solution depends on $\alpha$, $\varepsilon$ and $c$. In order to check the order of the numerical scheme,  we define $e^{(n)}$ the {\em relative $\ell^2$-error} at time $t=n\delta t$ given by:
\begin{equation*}
e^{(n)}=\left\|u_{\rm ref}(t_n,\cdot)-u^{n}(\cdot) \right\|_2/\left\| u_{\rm ref}(t_n,\cdot)\right\|_2,
\end{equation*}
where $u^n$ is the solution to the numerical scheme and where we use trapezoidal rule to compute the $\ell^2$-norm. Thanks to the definition of $e^{(n)}$, we consider the error function given by 
the maximum of $e^{(n)}$ with respect to $0<n\leq N$
\begin{equation*}
\mathcal{E}_P=\max_{0<n\leq N} \left(e^{(n)}\right)
\end{equation*}
which corresponds to the discrete version of $L^\infty_t L^2_x$ error function. Since we consider the Crank-Nicolson scheme \eqref{c-cn}, we should have the bound
\begin{equation}
  \label{eq:ccn_bound}
\mathcal{E}_P \leq  C_t\delta t^2+C_x\delta x^2.  
\end{equation}
We consider two kinds of initial conditions respectively of Gaussian type and modulated Gaussian (or  wave packet). The two initial conditions we consider are
$$
u_{0,G}=\exp\left(-400\left(x-\frac{1}{2}\right)^2\right), \qquad u_{0,WP}=u_{0,G}\sin(20\pi x).
$$
The evolutions of the solutions for this two initial data and for conditions $(\alpha=c=0,\varepsilon=10^{-3})$, $(c=0,\alpha=\varepsilon=10^{-3})$ and $(\alpha=\varepsilon=10^{-3},c=2)$ are plotted respectively on Figures \ref{fig:Ex1}, \ref{fig:Ex2} and \ref{fig:Ex3}.
\begin{figure}[htbp]
\begin{center}
\begin{tabular}{cc}
\includegraphics[width=0.48\textwidth]{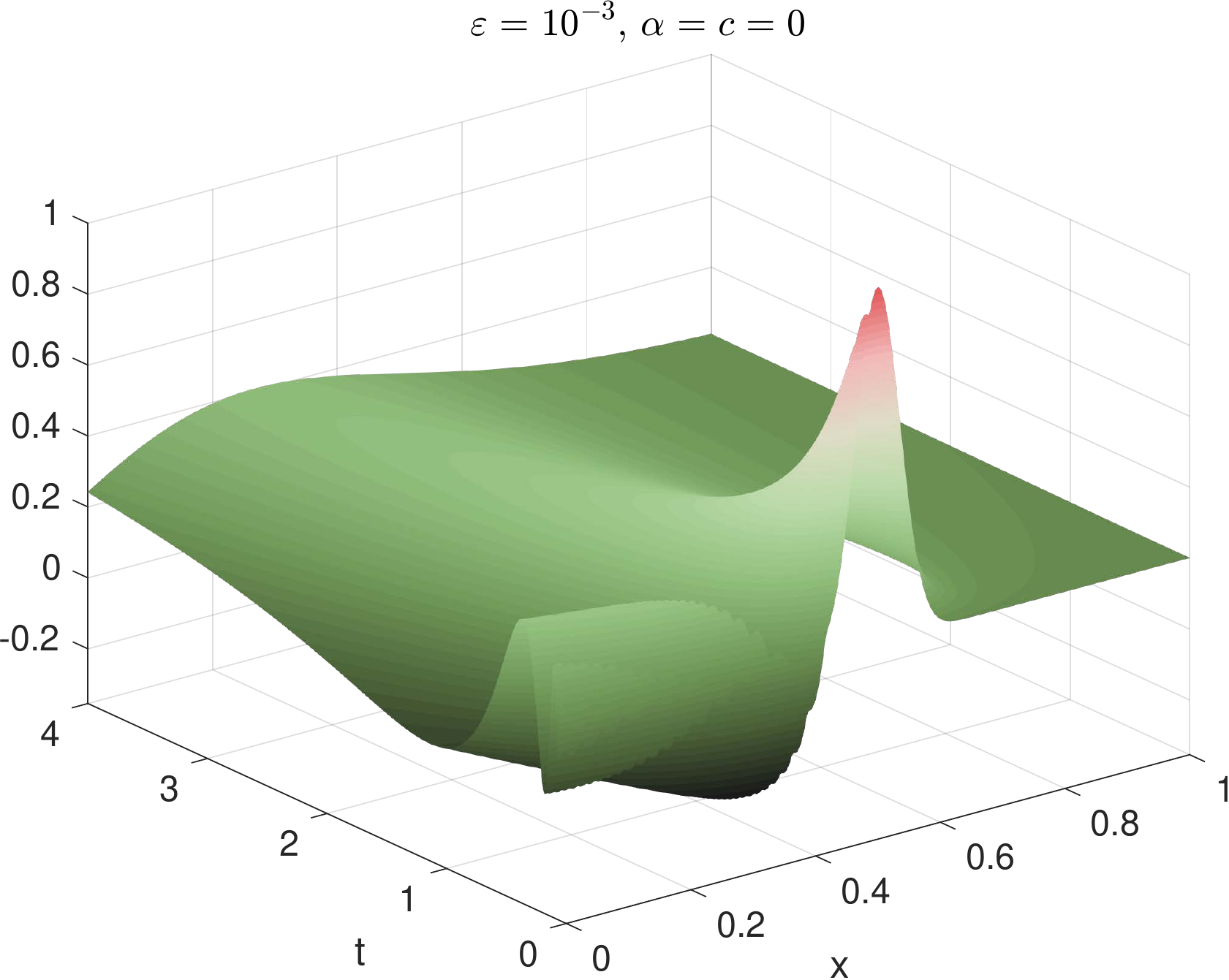} & \includegraphics[width=0.48\textwidth]{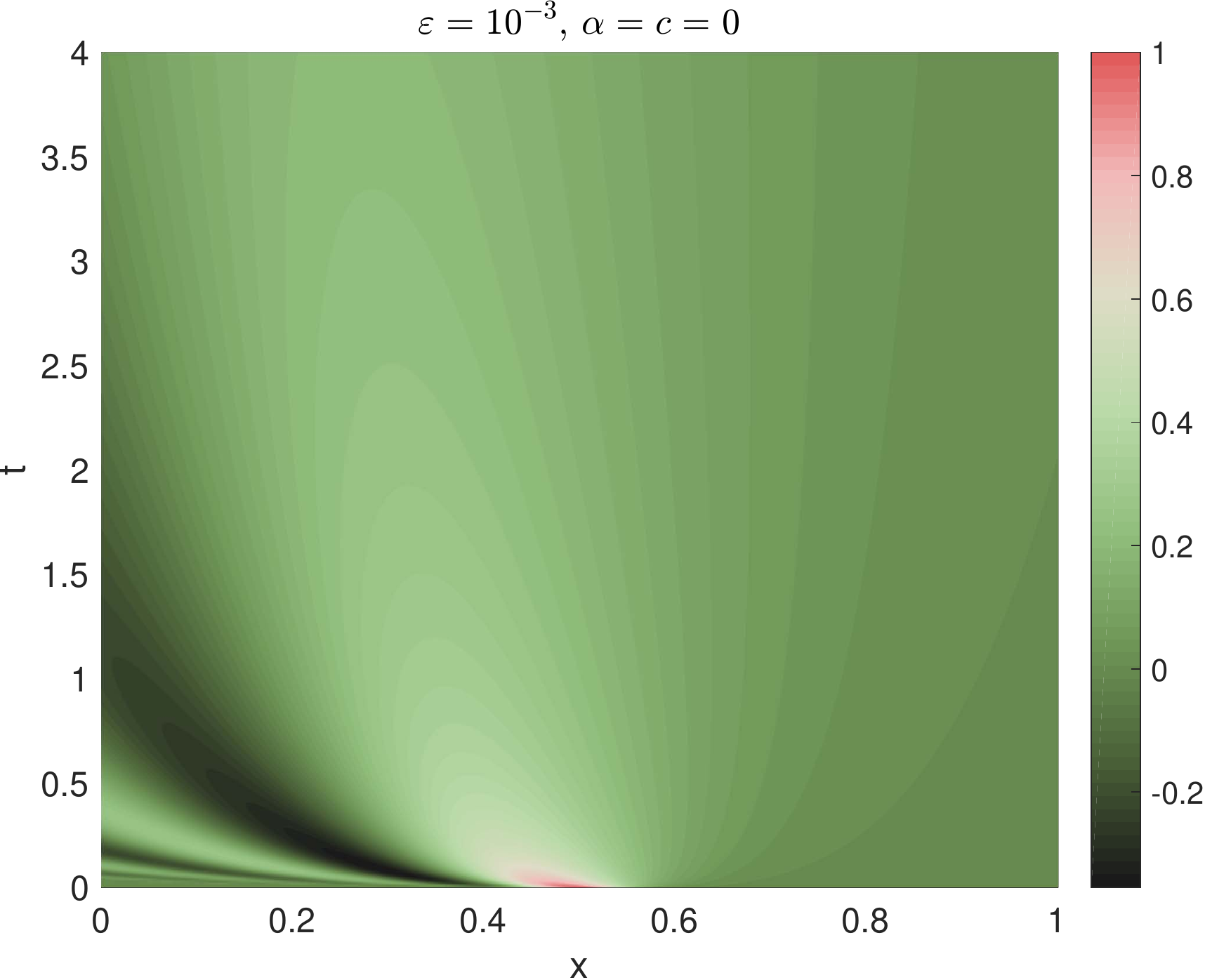}
\end{tabular}
\end{center}  
  \caption{Evolution of the reference solution for $(\alpha=c=0,\varepsilon=10^{-3})$ and $u_0=u_{0,G}$}
  \label{fig:Ex1}
\end{figure}
\begin{figure}[htbp]
\begin{center}
\begin{tabular}{cc}
\includegraphics[width=0.48\textwidth]{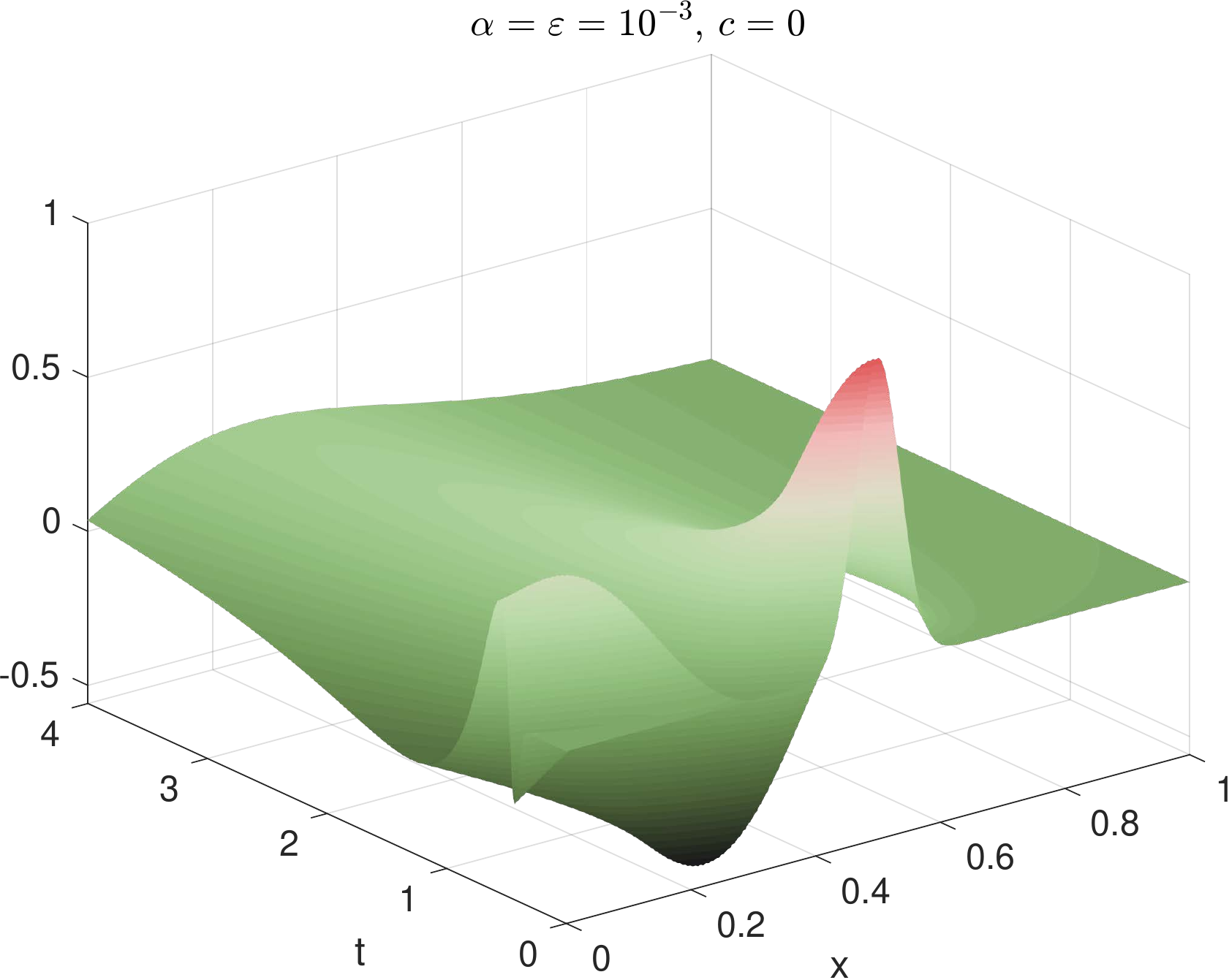} & \includegraphics[width=0.48\textwidth]{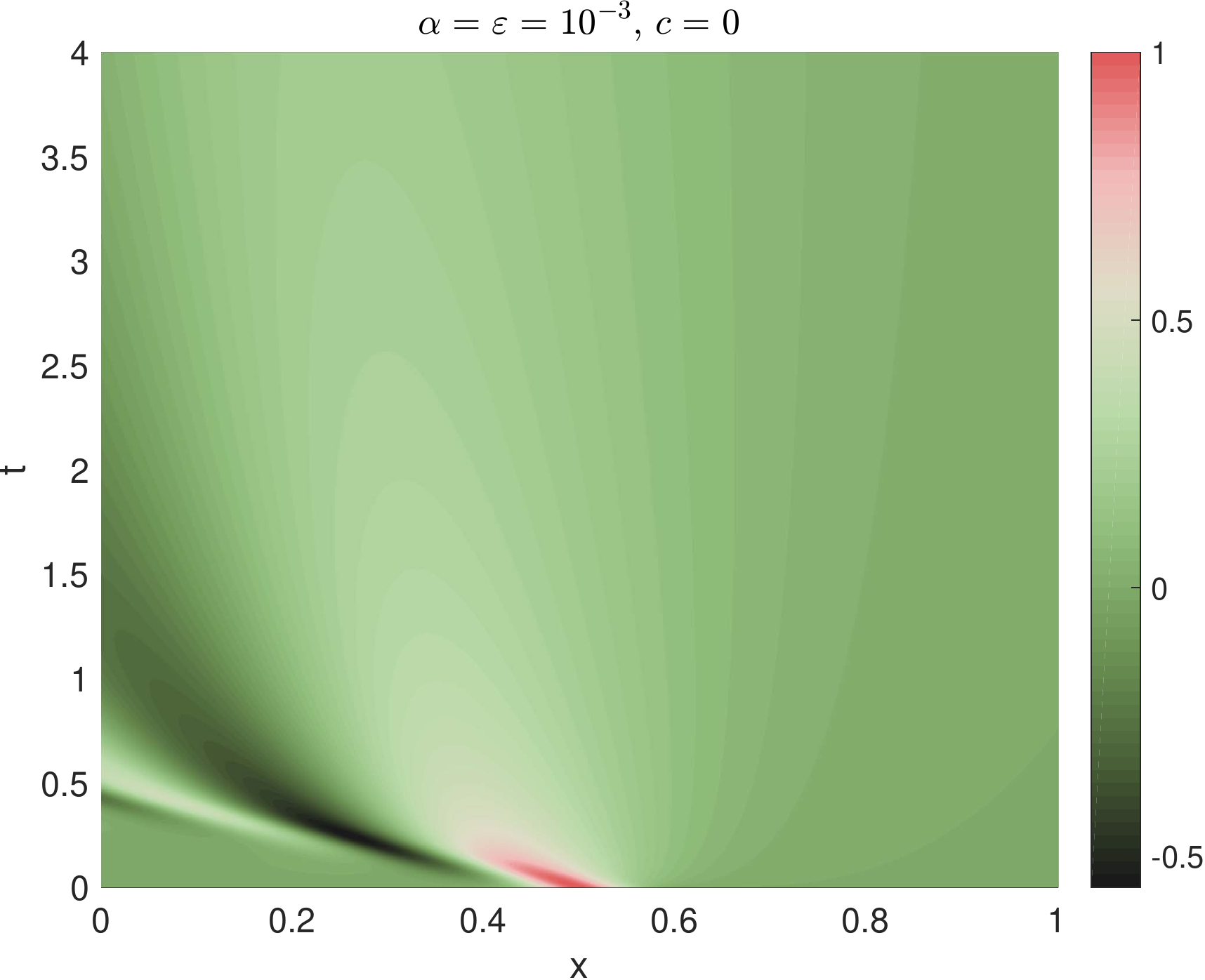}
\end{tabular}
\end{center}  
  \caption{Evolution of the reference solution for $(c=0,\alpha=\varepsilon=10^{-3})$ and $u_0=u_{0,G}$}
  \label{fig:Ex2}
\end{figure}
\begin{figure}[htbp]
\begin{center}
\begin{tabular}{cc}
\includegraphics[width=0.48\textwidth]{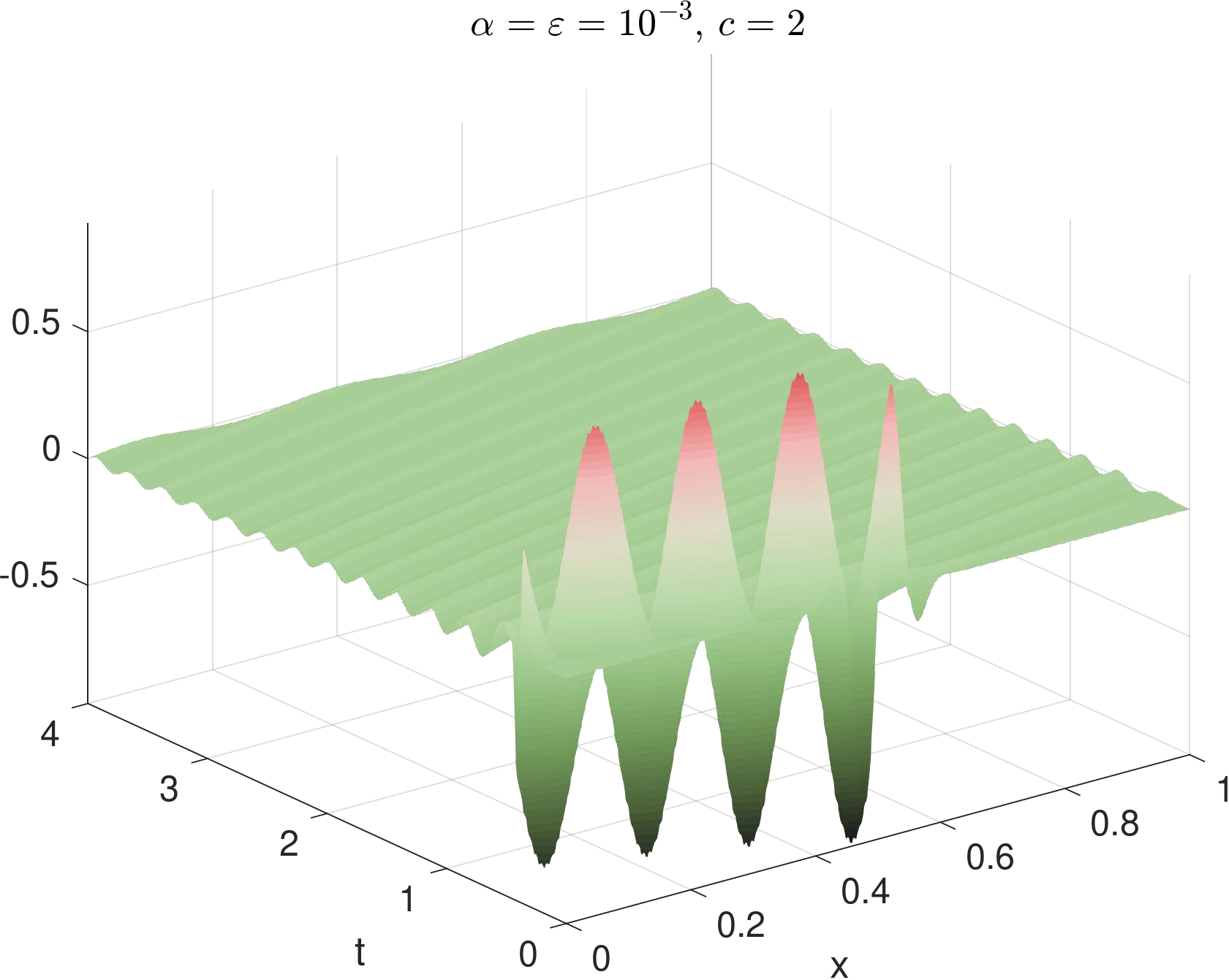} & \includegraphics[width=0.48\textwidth]{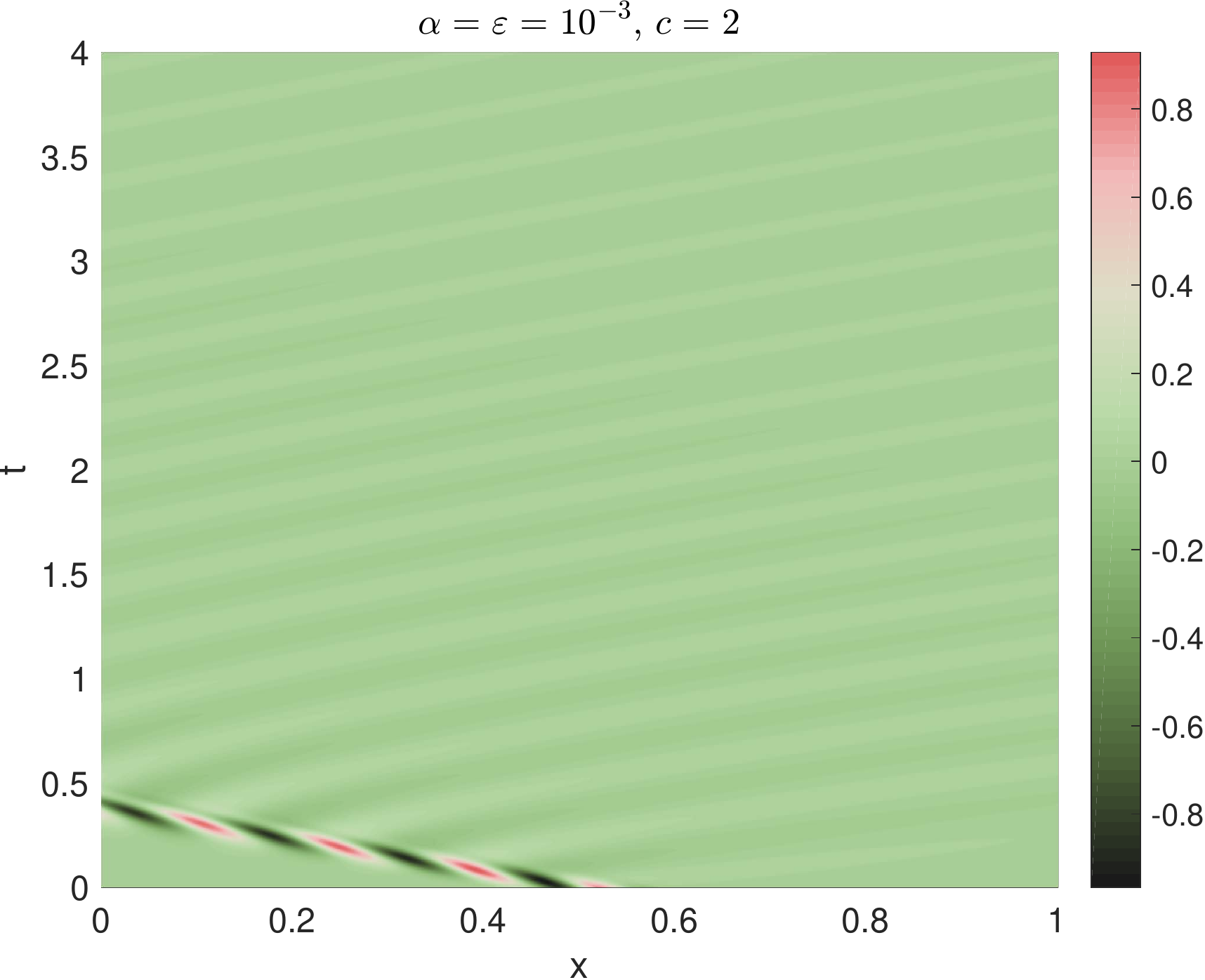}
\end{tabular}
\end{center}  
  \caption{Evolution of the reference solution for $(c=2,\alpha=\varepsilon=10^{-3})$ and $u_0=u_{0,WP}$}
  \label{fig:Ex3}
\end{figure}
We plot on Figure \ref{fig:err_wrt_dx} the behavior of $\mathcal{E}_P$ with respect to $\delta x$ for various $\delta t$ for the three test cases. In all cases and for $\delta x>5\cdot 10^{-5}$, we recover the second order behaviour of the numerical scheme. There exists a saturation process linked to $C_t$. When $\delta x$ is small enough, the dominating term in \eqref{eq:ccn_bound} is $C_t\delta t^2$. When $\delta x<5\cdot 10^{-5}$, the behaviour of $\mathcal{E}_P$ is deteriorated and the relation \eqref{eq:ccn_bound} is not valid anymore. This process is linked with a $(\delta x,\delta t)$-singularity of the convolution coefficients $\tilde{s}^u$, $\tilde{p}^u$, $\tilde{s}^s$ and $\tilde{p}^s$. 
Indeed, as already mentioned in the previous section, our strategy to compute these coefficients is based on the inversion of a $4\times 4$ matrix: as $\delta x\to 0$, one shows that its determinant is of order $\displaystyle O\left(\frac{c\delta x^2}{\varepsilon}+\frac{\delta x^3}{\varepsilon\delta t}\right)$ which increases the numerical errors in the computation of convolution coefficients. 
This bad behaviour is however limited when $\delta t>5\cdot 10^{-5}$. A way to correct the $\delta x$-singularity is proposed in the following subsection. The upper-left subfigure in Figure \ref{fig:err_wrt_dx} has to be compared to Figure 5 in \cite{BEL-V} which was limited to $\delta x \approx 10^{-3}$ due to the unstable procedure of inverse $\mathcal{Z}$-transform. Moreover, the study of the error $\mathcal{E}_P$ for very small $\delta x$ and $\delta t$ seems to have  never been produced before in the literature (for example, the smallest $\delta x$ is approximately $10^{-3}$ with $\delta t=10^{-4}$ in \cite{AABES})  and may be present for other transparent boundary conditions and other equations.
\begin{figure}[htbp]
\begin{center}
\begin{tabular}{cc}
\includegraphics[width=0.48\textwidth]{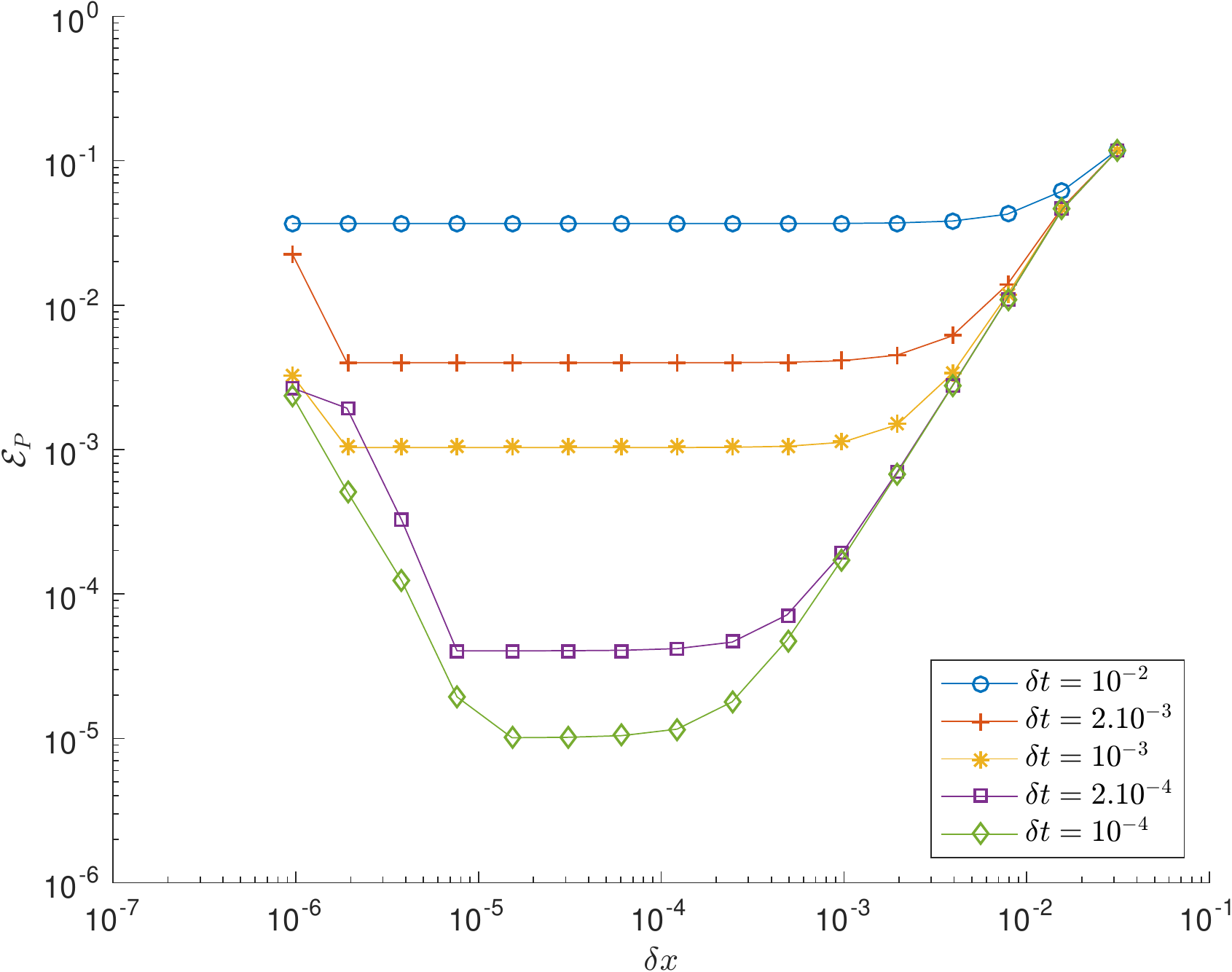} & \includegraphics[width=0.48\textwidth]{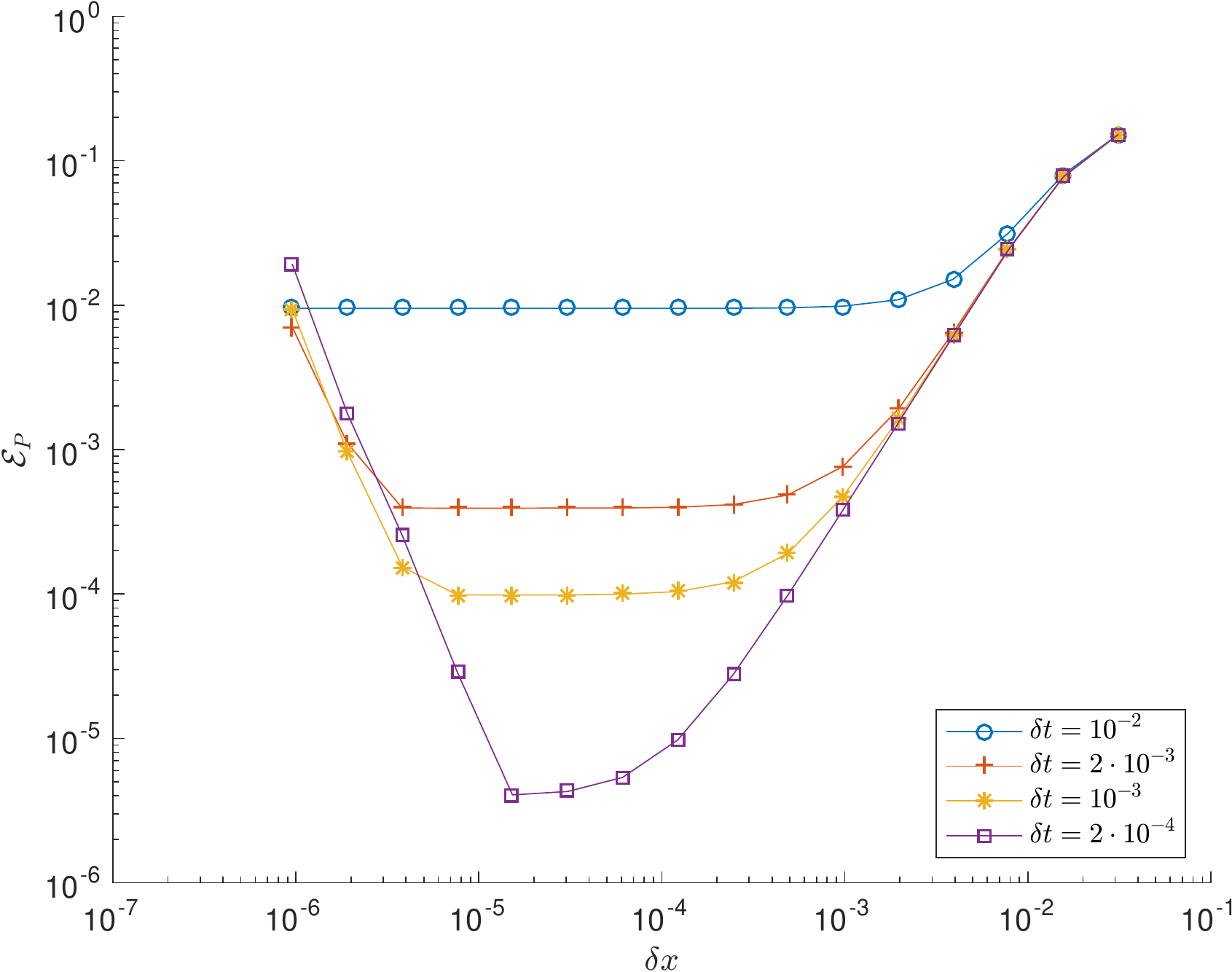}\\
$(\alpha=c=0,\varepsilon=10^{-3})$, $u_0=u_{0,G}$ & $(c=0,\alpha=\varepsilon=10^{-3})$, $u_0=u_{0,G}$\\[2mm]
\includegraphics[width=0.48\textwidth]{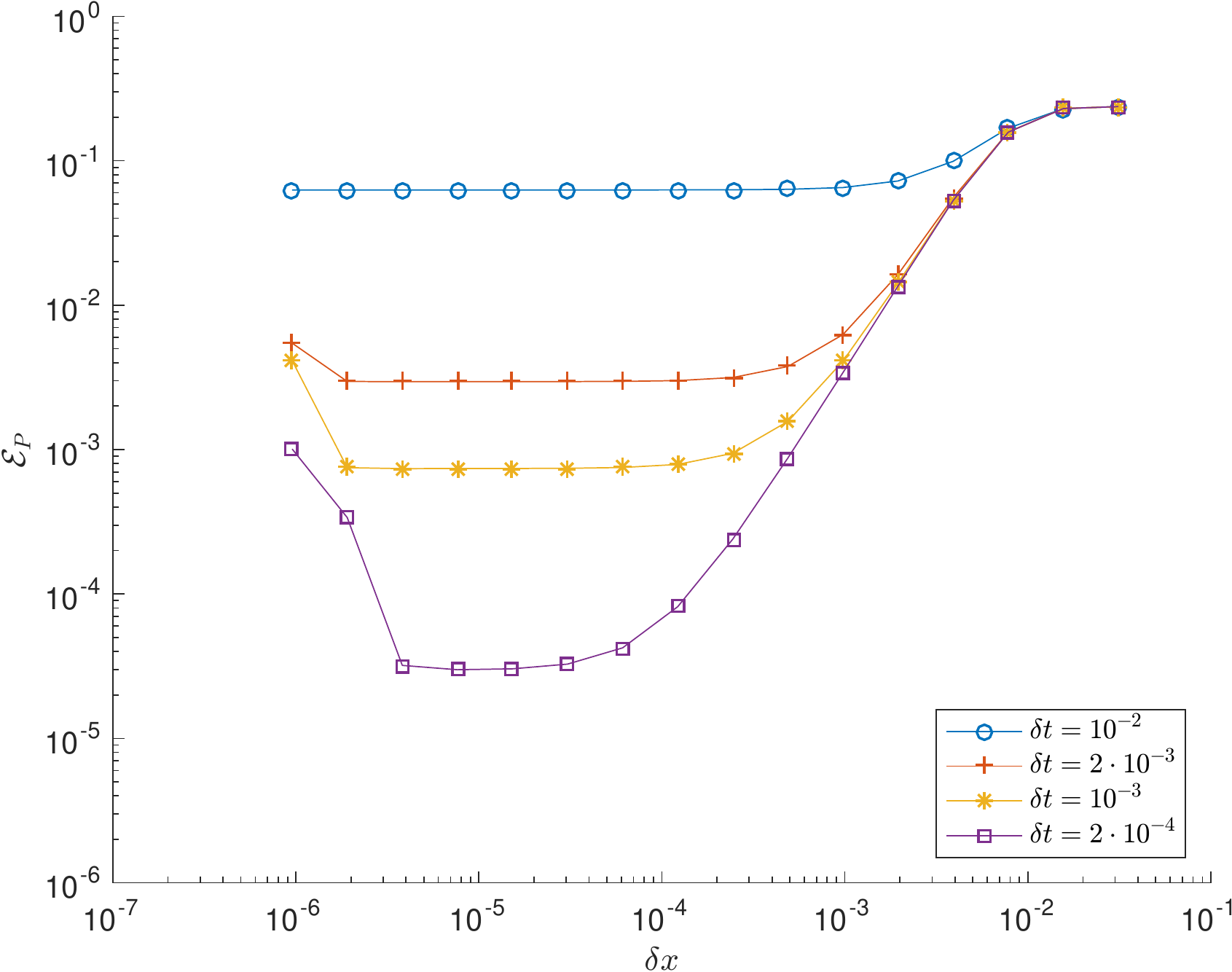} & \\
$(c=2,\alpha=\varepsilon=10^{-3})$,$u_0=u_{0,WP}$ &
\end{tabular}
\end{center}  
  \caption{Evolution of $\mathcal{E}_{P}$ with respect to $\delta x$ for various $\delta t$.}
  \label{fig:err_wrt_dx}
\end{figure}
We also plot the evolution of $\mathcal{E}_{P}$ with respect to $\delta t$ with $\delta x=2^{-14}\approx 6\cdot 10^{-5}$ for $(c=0,\alpha=\varepsilon=10^{-3})$ and $u_0=u_{0,G}$ on Figure \ref{fig:err_wrt_dt} (the results for other test cases are similar). The second-order with respect to $\delta t$ is well recovered.
\begin{figure}[htbp]
\begin{center}
\includegraphics[width=0.48\textwidth]{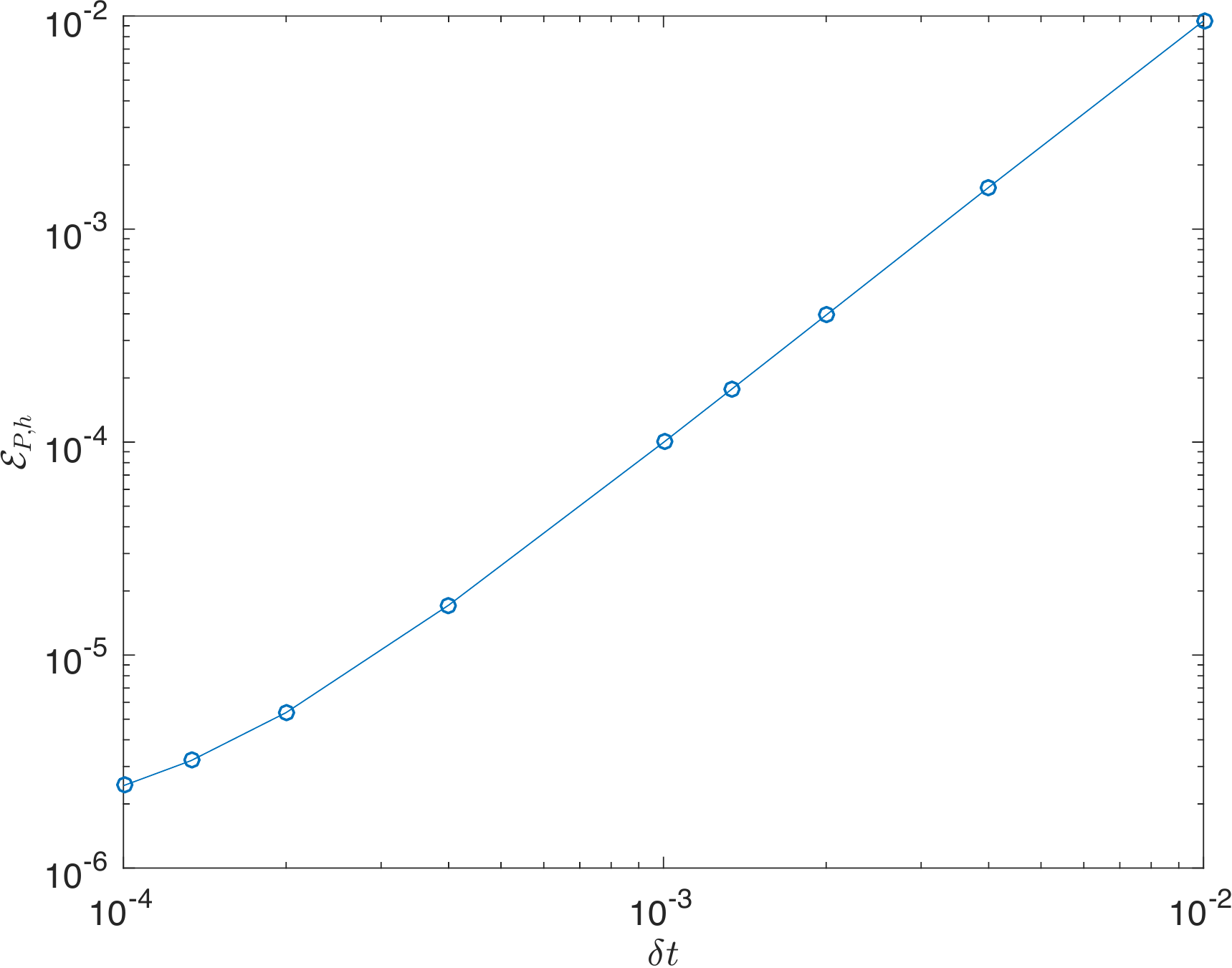} 
\end{center}  
  \caption{Evolution of $\mathcal{E}_{P}$ with respect to $\delta t$ for $\delta x=2^{-14}$.}
  \label{fig:err_wrt_dt}
\end{figure}

\subsection{Case 2: approximate discrete transparent boundary conditions \label{asympt_cof}}

In this section, we explore the limit $\delta x\to 0$. In order to simplify the discussion, we focus on the (lKdV) equation ($\alpha=0$). The general case is presented in Appendix. We first derive an asymptotic expansion of the coefficients involved in the formulation of the discrete transparent boundary conditions (\ref{left-bc-n}) and (\ref{right-bc-n}). Then, we present some numerical results. In particular, we present convergence results to verify that the truncation procedure does not introduce numerical instabilities and does not destroy the order of consistency of the numerical scheme.

Recall that the problem of inverting the $\mathcal{Z}$-transform in transparent boundary conditions (\ref{left-bc}) and (\ref{right-bc}) amounts to expand into Laurent series the functions $s^s(z), s^u(z), p^s(z), p^u(z)$ defined by the relation
\[
\begin{array}{lcl}
P(r) & = & r^4 - 2 r^3 +\frac{4 \delta x^3}{\varepsilon\delta t} p(z) r^2 + 2r-1\\
& = & \left(r^2-s^s(z) r+p^s(z)\right)\left(r^2- s^u(z) r+ p^u(z)\right).
\end{array}
\]
\noindent
The roots of $r^2 - s^s r + p^s$ belongs to $\{ r \in \mathbb{C}, \ |r|<1\}$ whereas the ones of $r^2 - s^u r + p^u$ belongs to $\{ r \in \mathbb{C}, \ |r|>1\}$.
Let us calculate $(s^s,p^s,s^u,p^u)$. These functions satisfy
\begin{equation}\label{eqvp}
\left \{
\begin{array}{lcl}
s^s + s^u & = & 2 ,\\
s^s s^u + p^s +p^u & = & \frac{4\delta x^3}{\varepsilon \delta t} p(z), \\
s^s p^u + s^u p^s & = & -2, \\ 
p^s p^u & = & -1.
\end{array}
\right.
\end{equation}
We look for an asymptotic expansion of these quantities as $\dx \to 0$ in the form:
\[
s^s = \sum_{k\ge 0} s_k \dx^{k}, \ \ p^s = \sum_{k\ge 0} p_k \dx^{k},\ \ s^u = \sum_{k\ge 0} t_k \dx^{k}, \ \ p^u = \sum_{k\ge 0} q_k \dx^{k}.
\]
By inserting this  expansion into (\ref{eqvp}) and identifying $O(\delta x^n)$ terms with $(n\in\mathbb{N})$, we obtain a non linear system and a series of linear systems to be solved. First, by identifying { $0^\text{th}$ order terms}, one finds the nonlinear system of equations:
\[
\left \{
\begin{array}{lcl}
s_0 +t_0 & = & 2, \\
s_0t_0+p_0+ q_0 & = & 0, \\
s_0q_0+t_0p_0 & = & -2,\\
p_0q_0 & = & -1.
\end{array}
\right.
\]
The solution writes $(s_0,p_0,t_0,q_0) = (0,-1,2,1)$. Next, we identify $O(\delta x^n)$ terms with $n\geq 1$. One finds the family of linear systems:

\[
A \begin{pmatrix}
s_n \\ p_n \\ t_n \\q_n
\end{pmatrix} = F_n =  \Sigma_n - G_n \ \text{ where } A=\begin{pmatrix}
1 & 0 & 1 & 0 \\ t_0 & 1 & s_0 & 1 \\ q_0 & t_0 & p_0 & s_0 \\ 0 & q_0 & 0 & p_0
\end{pmatrix} = \begin{pmatrix}
1 & 0 & 1 & 0 \\ 2 & 1 & 0 &  1 \\ 1 & 2 & -1 & 0 \\ 0 & 1 & 0 & -1
\end{pmatrix},
\]
\[
\Sigma_n = \begin{pmatrix}
0 \\ 0 \\ 0 \\0
\end{pmatrix} \text{ if } n \neq 3,\:\:
 \ \Sigma_3 = \begin{pmatrix}
0 \\ \frac{4}{\varepsilon \delta t}p(z)  \\ 0 \\0
\end{pmatrix}, \ G_n = \begin{pmatrix}
0 \\ \sum_{k=1}^{n-1}s_k t_{n-k} \\ \sum_{k=1}^{n-1} s_k q_{n-k} + \sum_{k=1}^{n-1} t_k p_{n-k} \\ \sum_{k=1}^{n-1} p_k q_{n-k}
\end{pmatrix}.
\]
The matrix $A$ is not invertible, the eigenvalue $0$ is simple and associated to $v=\begin{pmatrix}1 \\ -1 \\ - 1 \\ -1 \end{pmatrix}$. If the condition 
$$
\displaystyle
\det\left (F_n,\begin{pmatrix} 0 \\ 1 \\ 2 \\1 \end{pmatrix}, \begin{pmatrix} 1 \\ 0 \\ -1 \\ 0 \end{pmatrix}, \begin{pmatrix}
0 \\ 1 \\ 0 \\ -1 \end{pmatrix} \right )=0$$ 
is fulfilled, then $U_n =(s_n,p_n,t_n,q_n)^T$ is given by
\[
U_n = \lambda_n v + \frac{(F_n)_2+(F_n)_4}{2} e_2 + \left ((F_n)_2-(F_n)_3+(F_n)_4 \right ) e_3+ \frac{F_2-F_4}{2}e_4 = \left (\begin{array}{rcl}
\lambda_n & & \\ - \lambda_n & +& \dsp \frac{(F_n)_2+(F_n)_4}{2} \\ -\lambda_n & +& \dsp (F_n)_2-(F_n)_3+(F_n)_4 \\ - \lambda_n & + & \dsp \frac{F_2-F_4}{2}
\end{array}\right ),
\]
where $(e_1,e_2,e_3,e_4)$ is the canonical basis of $\R^4$. Let $\lambda_1$ the root of $\lambda_1^3 + \frac{2}{\varepsilon\delta t}p(z)=0$ 
whose real part is negative. We get:
\[
s^s  = \dsp  \lambda_1  \dx + \frac{\lambda_1^2}{2} \dx^2 +\frac{p}{3 \varepsilon\delta t} \dx^3 + O(\dx^4),
\]
\[
s^u = 2 - \lambda_1 \dx - \frac{\lambda_1^2}{2}  \dx^2 - \frac{p}{3 \varepsilon\delta t} \dx^3 + O(\dx^4) ,
\]
\[
p^s =  -1 - \lambda_1 \dx - \frac{\lambda_1^2}{2} \dx^2 + \frac{2 p}{3 \varepsilon\delta t}  \dx^3 + O(\dx^4),
\]
\[
p^u =  1 - \lambda_1 \dx + \frac{\lambda_1^2}{2} \dx^2  + \frac{2 p}{3 \varepsilon\delta t}  \dx^3 + O(\dx^4).
\]
We now need to invert the $\mathcal{Z}$ transform of $z\mapsto \lambda_1(s(z)) = - \left(\frac{2}{\varepsilon\delta t}\right)^{1/3} p(z)^{1/3}$.
\noindent
Note that 
$$
\displaystyle
p(z)^{k/3}=\frac{(1-z^{-1})^{k/3}}{(1+z^{-1})^{k/3}},\quad \forall |z|>1,\quad \forall k\in\mathbb{Z}.
$$
\noindent
As a consequence, $p(z)^{k/3}$ can be expanded into Laurent series explicitly:   indeed, $(1-z^{-1})^{\gamma}$ and $(1+z^{-1})^{\gamma}$ expand as
$$
\begin{array}{ll}
\displaystyle
(1-z^{-1})^{\gamma}=\sum_{p=0}^{\infty}\frac{\alpha_p^{(\gamma)}}{z^p},\quad \alpha_{p+1}^{(\gamma)}=-\frac{\displaystyle \gamma	-(p-1)}{p}\alpha_p^{(\gamma)},\quad\alpha_0=1,\\
\displaystyle
(1+z^{-1})^{\gamma}=\sum_{p=0}^{\infty}\frac{\beta_p^{(\gamma)}}{z^p},\quad \beta_{p+1}^{(k)}=\frac{\displaystyle \gamma-(p-1)}{p}\beta_p^{(\gamma)},\quad\beta_0=1.
\end{array}
$$\noindent
This, in turn, provides an {\it explicit} expansion of $\lambda_1(s(z))$ and $(\lambda_1^2(s(z)))^2$ into Laurent series
\begin{equation}\label{dl-lambda1}
\displaystyle
\lambda_1(s(z))= \sum_{p=0}^{\infty}\frac{\sigma_p^{(1)}}{z^p},\quad
(\lambda_1(s(z)))^2= \sum_{p=0}^{\infty}\frac{\sigma_p^{(2)}}{z^p}.
\end{equation}
where
\[
\sigma_p^{(1)} = - \left ( \frac{2}{\varepsilon \delta t} \right )^{1/3} \sum_{l=0}^n \alpha_l^{(1/3)} \beta_{n-l}^{(-1/3)}, \quad \sigma_p^{(2)} = \left ( \frac{2}{\varepsilon \delta t} \right )^{2/3} \sum_{l=0}^n \alpha_l^{(2/3)} \beta_{n-l}^{(-2/3)}
\]
\noindent
We are now in a position to formulate approximate discrete transparent boundary conditions. The transparent boundary conditions are written in term of asymptotic coefficients $\widetilde{as}^u$, $\widetilde{ap}^us$, $\widetilde{as}^s$, $\widetilde{ap}^s$,as

\begin{equation}\label{bcr-n}
\begin{array}{l}
\dsp u_{J+1}^{n+1}+u_{J+1}^n - \sum_{k=0}^{n+1} u_{J}^k \widetilde{as}^s_{n+1-k} + \sum_{k=0}^{n+1} u_{J-1}^k \widetilde{ap}^s_{n+1-k} = 0, \\
\dsp u_{J+2}^{n+1}+u_{J+2}^n - \sum_{k=0}^{n+1} u_{J+1}^k \widetilde{as}^s_{n+1-k} + \sum_{k=0}^{n+1} u_{J}^k \widetilde{ap}^s_{n+1-k} = 0,
\end{array}
\end{equation}
\noindent
and
\begin{equation}\label{bcl-n}
\begin{array}{l}
\dsp u_0^{n+1}+u_0^n - \sum_{k=0}^{n+1} u_{-1}^k \widetilde{as}^u_{n+1-k} + \sum_{k=0}^{n+1} u_{-2}^k \widetilde{ap}^u_{n+1-k} = 0, \\
\dsp u_1^{n+1}+u_1^n - \sum_{k=0}^{n+1} u_0^k \widetilde{as}^u_{n+1-k} + \sum_{k=0}^{n+1} u_{-1}^k \widetilde{ap}^u_{n+1-k} = 0,
\end{array}\end{equation}
where
\[
\begin{array}{lcl}
\widetilde{as}_0^s & = & \dsp \sigma_0^{(1)} \dx + \frac{\sigma_0^{(2)}}{2}\dx^2 + \frac{\dx^3}{3 \varepsilon \delta t} +O(\dx^4),  \\
\widetilde{as}_1^s & =  & \dsp (\sigma_0^{(1)} + \sigma_1^{(1)}) \dx + \frac{\sigma_0^{(2)}+ \sigma_1^{(2)}}{2}\dx^2 - \frac{\dx^3}{3 \varepsilon \delta t} +O(\dx^4),   \\
\widetilde{as}_{p+1}^s & = & \dsp (\sigma_p^{(1)} + \sigma_{p+1}^{(1)}) \dx + \frac{ \sigma_p^{(2)} + \sigma_{p+1}^{(2)}}{2} \dx^2 +O(\dx^4), \quad p \ge 1,\\
\widetilde{ap}_0^s & = & \dsp -1 - \sigma_0^{(1)} \dx - \frac{\sigma_0^{(2)}}{2}\dx^2 + \frac{2 \dx^3}{3 \varepsilon \delta t} +O(\dx^4) , \\
\widetilde{ap}_1^s & =  & \dsp -1 -(\sigma_0^{(1)} + \sigma_1^{(1)}) \dx - \frac{\sigma_0^{(2)}+ \sigma_1^{(2)}}{2}\dx^2 - \frac{2\dx^3}{3 \varepsilon \delta t} +O(\dx^4)  , \\
\widetilde{ap}_{p+1}^s & = & \dsp -(\sigma_p^{(1)} + \sigma_{p+1}^{(1)}) \dx - \frac{ \sigma_p^{(2)} + \sigma_{p+1}^{(2)}}{2} \dx^2 +O(\dx^4), \quad p \ge 1, \\
\widetilde{as}_0^u & = & \dsp 2 - \sigma_0^{(1)} \dx - \frac{\sigma_0^{(2)}}{2}\dx^2 - \frac{\dx^3}{3 \varepsilon \delta t} +O(\dx^4) , \\
\widetilde{as}_1^u & =  & \dsp 2 -(\sigma_0^{(1)} + \sigma_1^{(1)}) \dx - \frac{\sigma_0^{(2)}+ \sigma_1^{(2)}}{2}\dx^2 + \frac{\dx^3}{3 \varepsilon \delta t} +O(\dx^4) ,  \\
\widetilde{as}_{p+1}^u & = & \dsp -(\sigma_p^{(1)} + \sigma_{p+1}^{(1)}) \dx - \frac{ \sigma_p^{(2)} + \sigma_{p+1}^{(2)}}{2} \dx^2 +O(\dx^4), \quad p \ge 1,\\
\widetilde{ap}_0^u & = & \dsp 1 - \sigma_0^{(1)} \dx + \frac{\sigma_0^{(2)}}{2}\dx^2 + \frac{2 \dx^3}{3 \varepsilon \delta t} +O(\dx^4) , \\
\widetilde{ap}_1^u & =  & \dsp 1 -(\sigma_0^{(1)} + \sigma_1^{(1)}) \dx + \frac{\sigma_0^{(2)}+ \sigma_1^{(2)}}{2}\dx^2 - \frac{2\dx^3}{3 \varepsilon \delta t} +O(\dx^4),   \\
\widetilde{ap}_{p+1}^u & = & \dsp -(\sigma_p^{(1)} + \sigma_{p+1}^{(1)}) \dx + \frac{ \sigma_p^{(2)} + \sigma_{p+1}^{(2)}}{2} \dx^2 +O(\dx^4), \quad p \ge 1.
\end{array}
\]
To illustrate numerically the efficiency of these new coefficients of convolution, we reproduce the test case $(\alpha=c=0,\varepsilon=10^{-3})$ with $u_0=u_{0,G}$ and compared the evolution of $\mathcal{E}_P$ on Figure \ref{fig:compar1}. We make use of these asymptotic coefficients only for small $\delta x$.
\begin{figure}[htbp]
\begin{center}
\begin{tabular}{cc}
\includegraphics[width=0.48\textwidth]{err_L2_eps_1em3_al_0_c_0_T_4_N_compar-crop.pdf} & \includegraphics[width=0.48\textwidth]{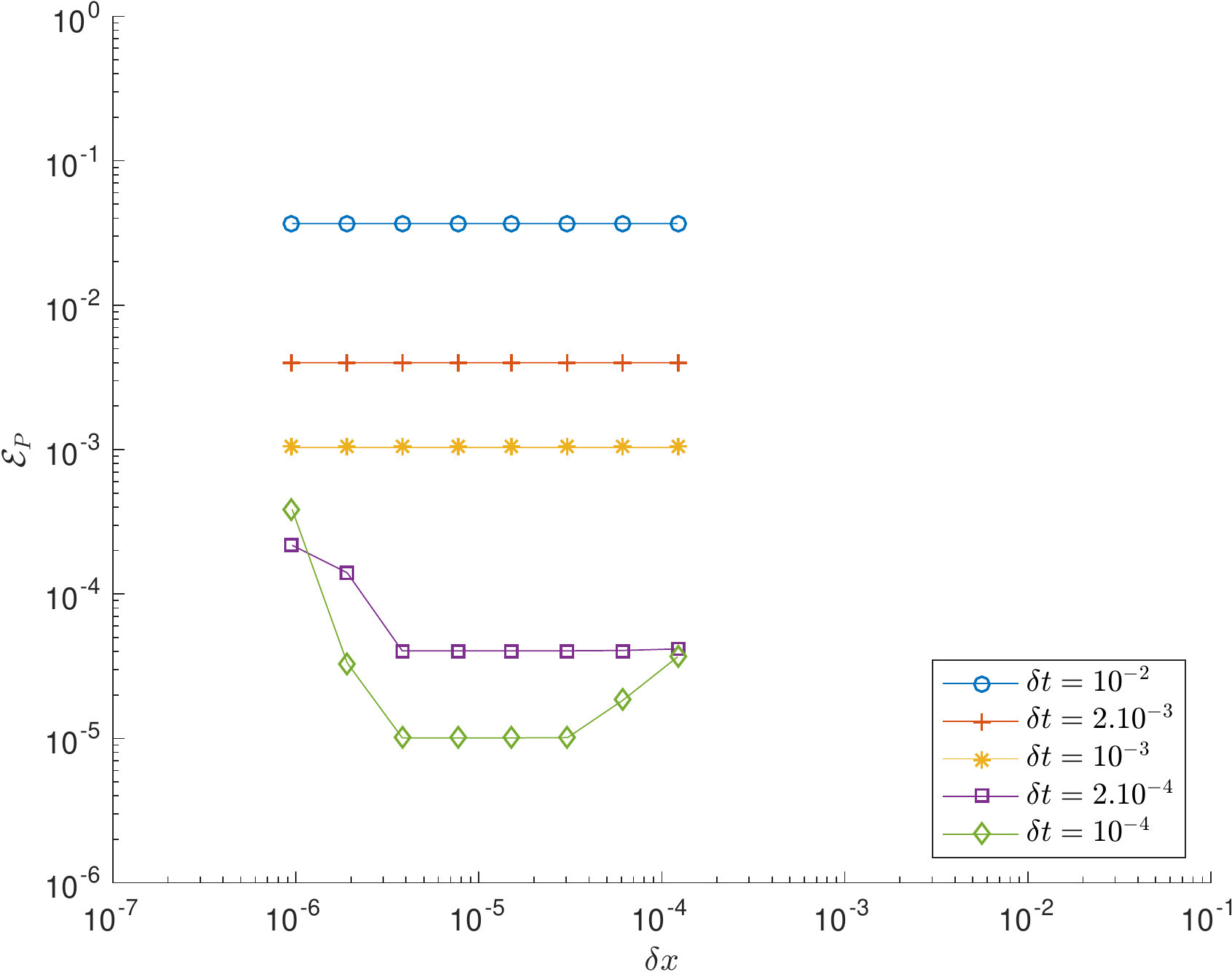}\\
$(\alpha=c=0,\varepsilon=10^{-3})$, $u_0=u_{0,G}$ & $(c=0,\alpha=\varepsilon=10^{-3})$, $u_0=u_{0,G}$\\
standard coefficients & asymptotic coefficients
\end{tabular}
\end{center}  
  \caption{Evolution of $\mathcal{E}_{P}$ with respect to $\delta x$ for various $\delta t$.}
  \label{fig:compar1}
\end{figure}
The bad behaviour of $\mathcal{E}_{P}$ is clearly limited  when $\delta x, \delta x^3/\delta t$ are very small. 
These asymptotic coefficients are also useful for long time simulations. We consider here $T=1000$, $\delta t=10^{-1}$, $\delta x=2^{-18}$ and $u_0=u_{0,G}$. We see on Figure \ref{fig:coef} that the standard coefficients do not have the good decay $n^{-3/2}$. This rate is clearly well preserved by the asymptotic coefficient. For this test case, we cannot compare the numerical solution to a reference solution. The two procedures described in subsection \ref{sub:ref} are not valid for such long time simulations. We therefore present here the evolution of the solution with standard and asymptotic coefficients but also of the discrete energy \eqref{eq:disc_energy} respectively on Figure \ref{fig:evol01} and \ref{fig:evol02}. It is clear that the behavior of the solution with standard coefficient is not good since the solution does not decay with $t>0$ and the discrete energy is growing. We obtain a good behavior with the asymptotic coefficients.

\begin{figure}[htbp]
\begin{center}
\begin{tabular}{cc}
\includegraphics[width=0.48\textwidth]{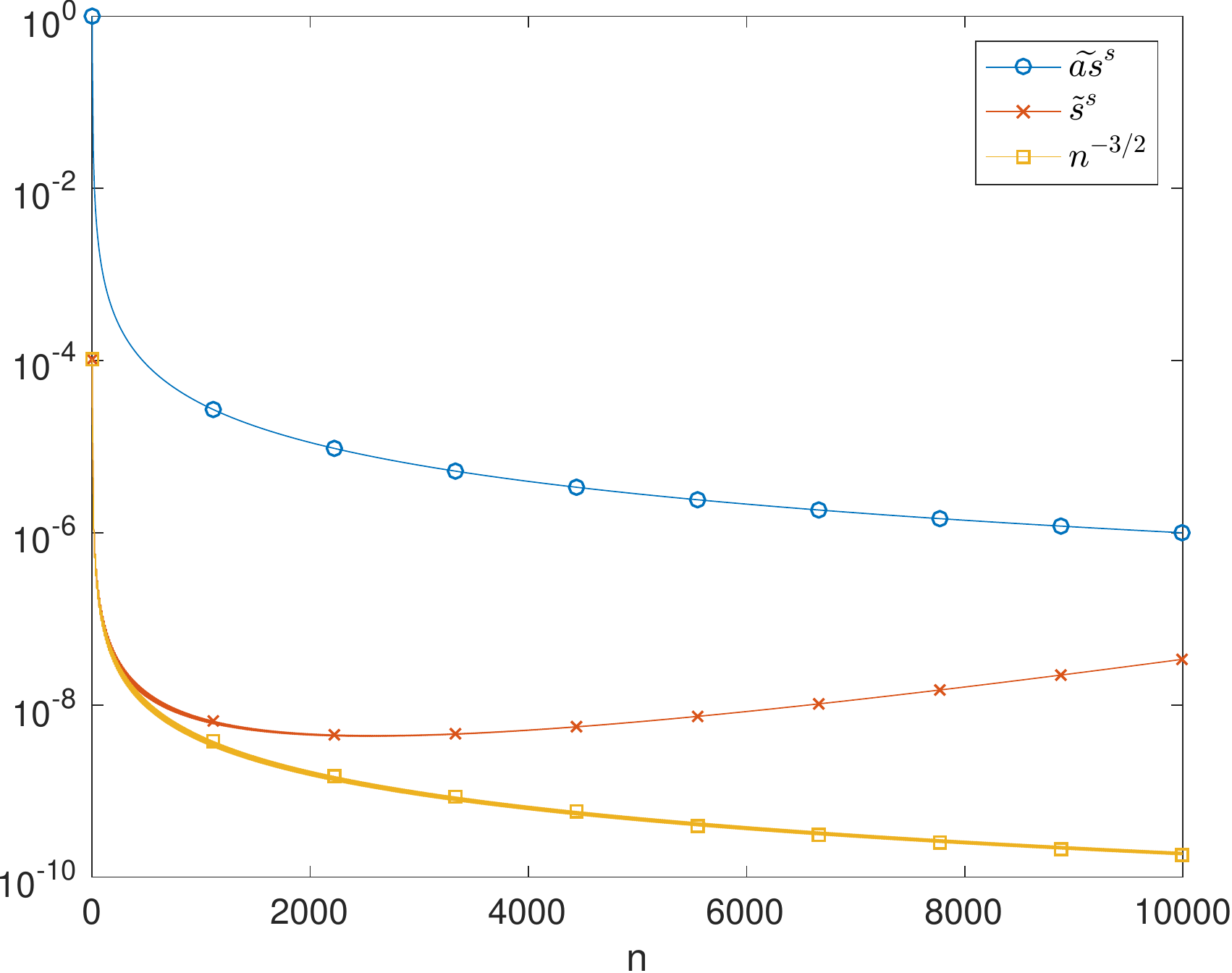} & \includegraphics[width=0.48\textwidth]{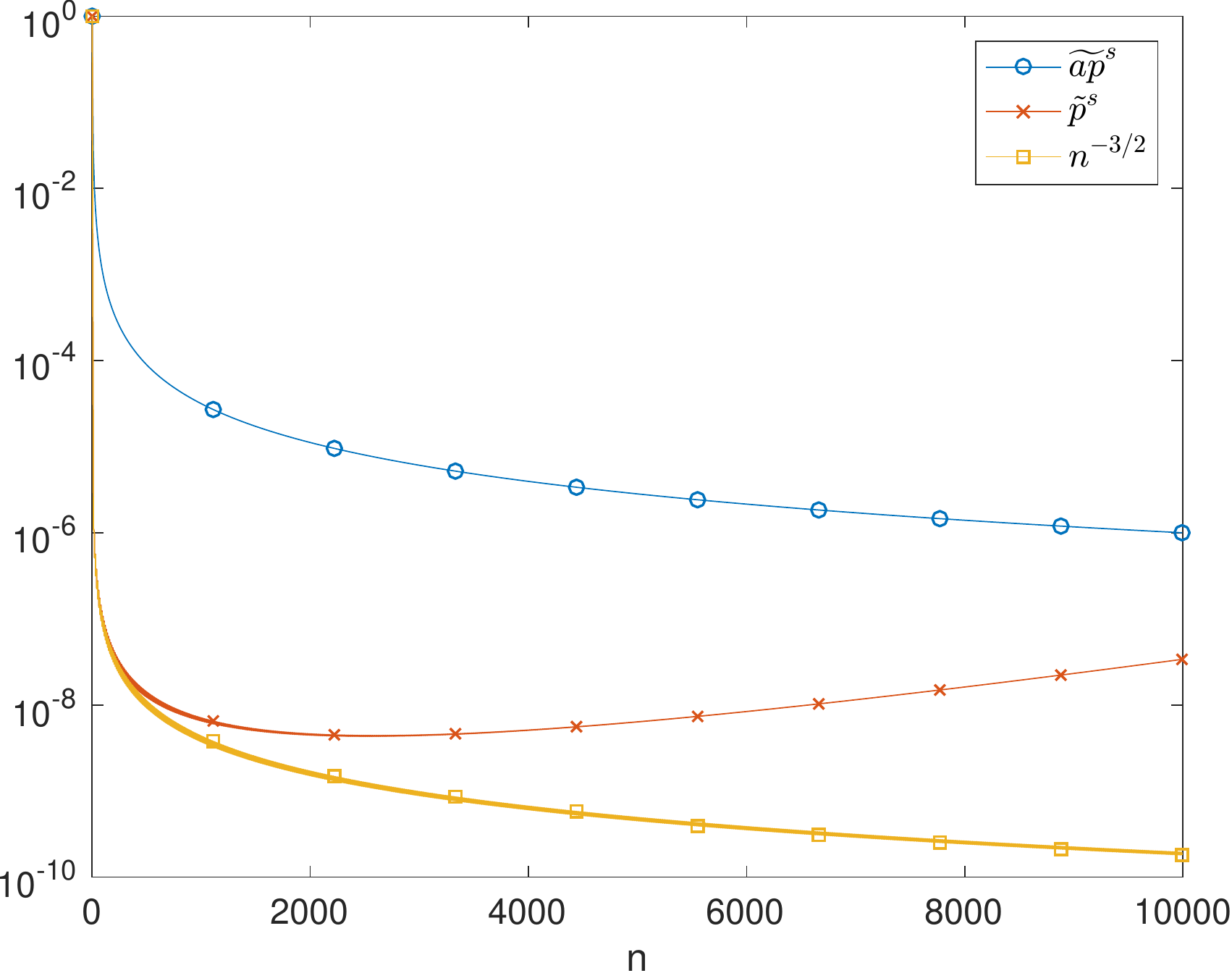}\\
coefficients $\tilde{p}^s$ and $\widetilde{ap}^s$ & coefficients $\tilde{s}^s$ and $\widetilde{as}^s$
\end{tabular}
\end{center}  
  \caption{Evolution of the convolution coefficients.}
  \label{fig:coef}
\end{figure}

\begin{figure}[htbp]
\begin{center}
\begin{tabular}{cc}
\includegraphics[width=0.48\textwidth]{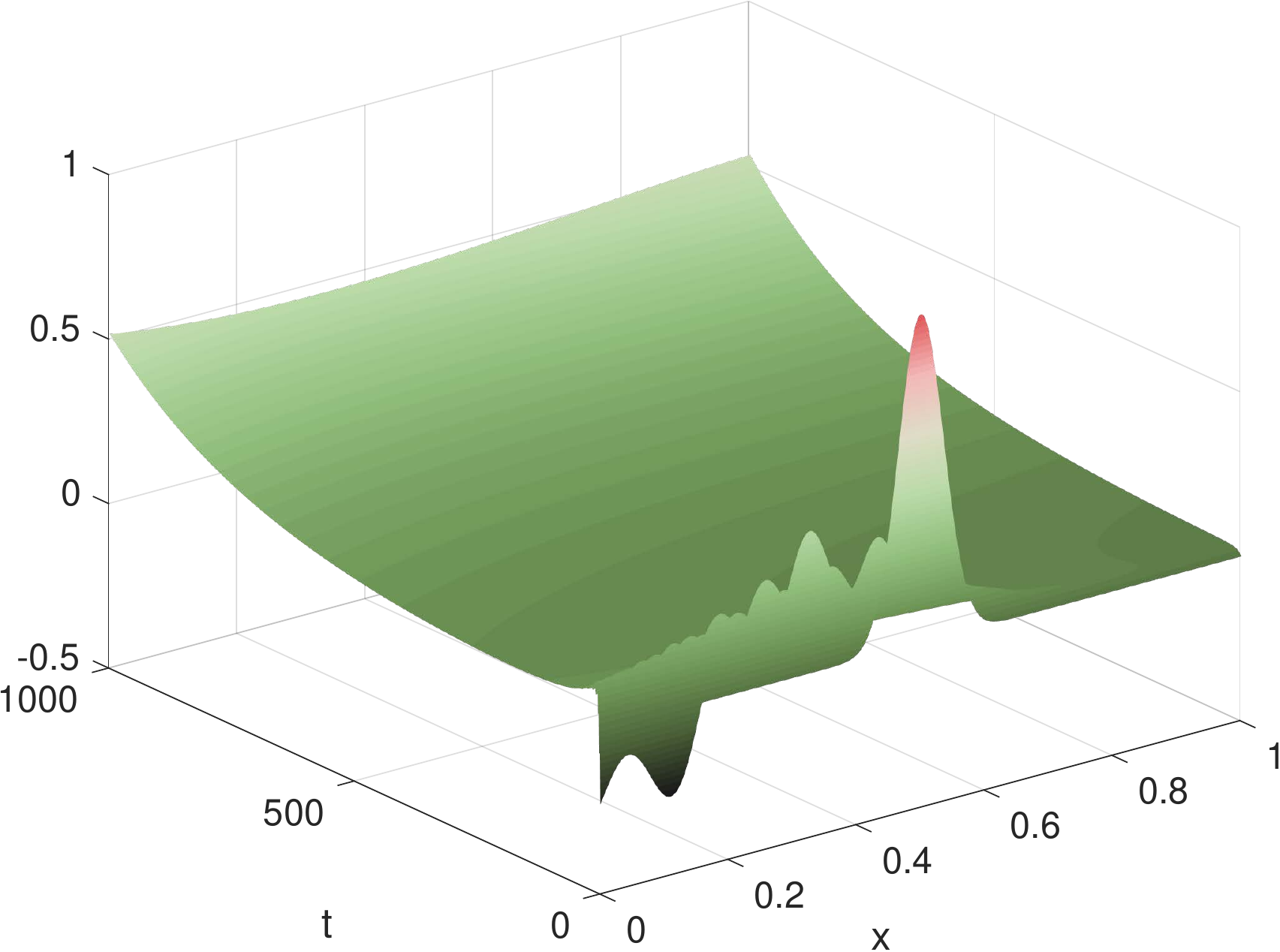} & \includegraphics[width=0.48\textwidth]{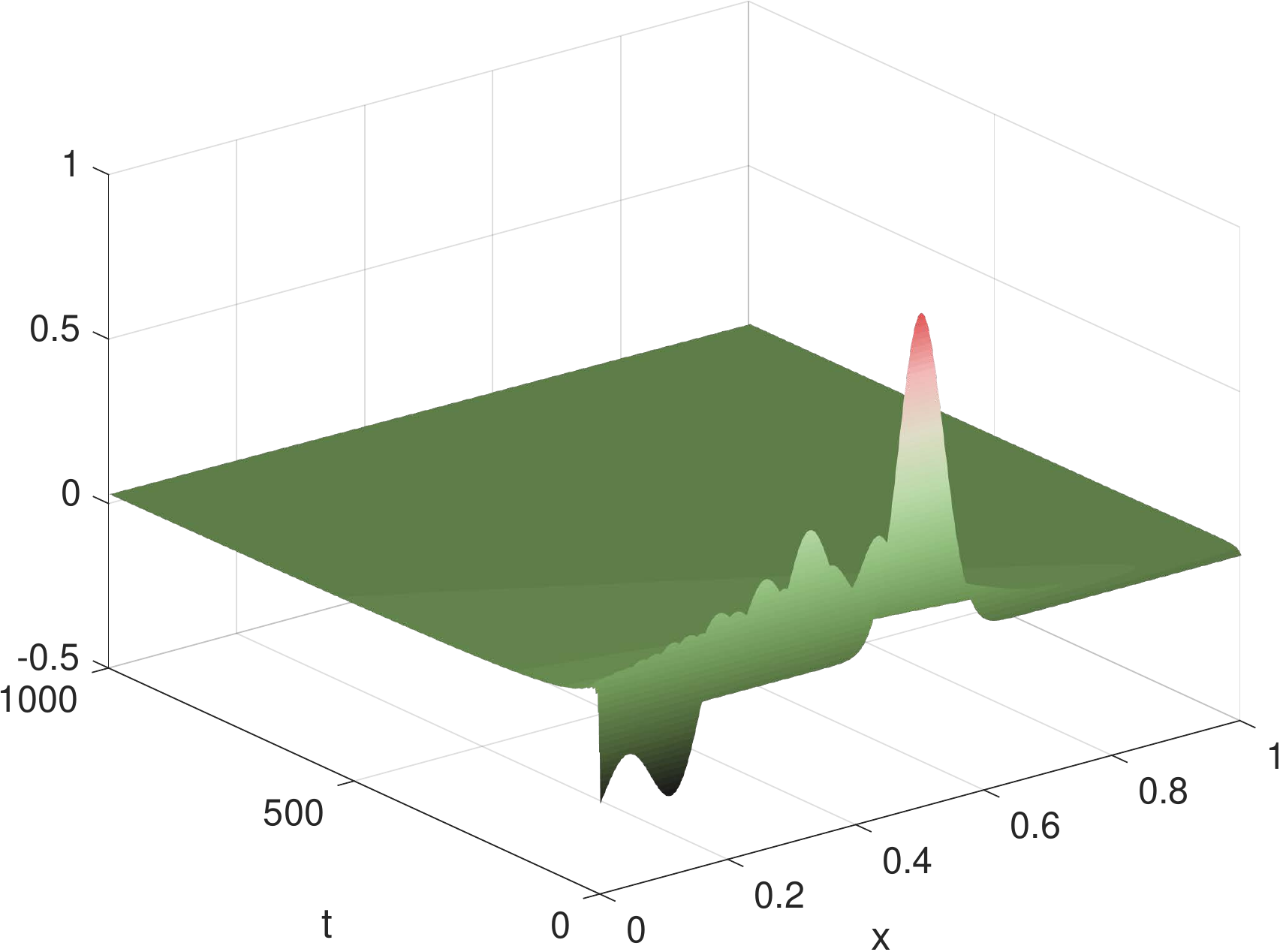}\\
standard coefficients & asymptotic coefficients
\end{tabular}
\end{center}  
  \caption{Evolution of the solution with standard and asymptotic convolution coefficients.}
  \label{fig:evol01}
\end{figure}

\begin{figure}[htbp]
\begin{center}
\includegraphics[width=0.48\textwidth]{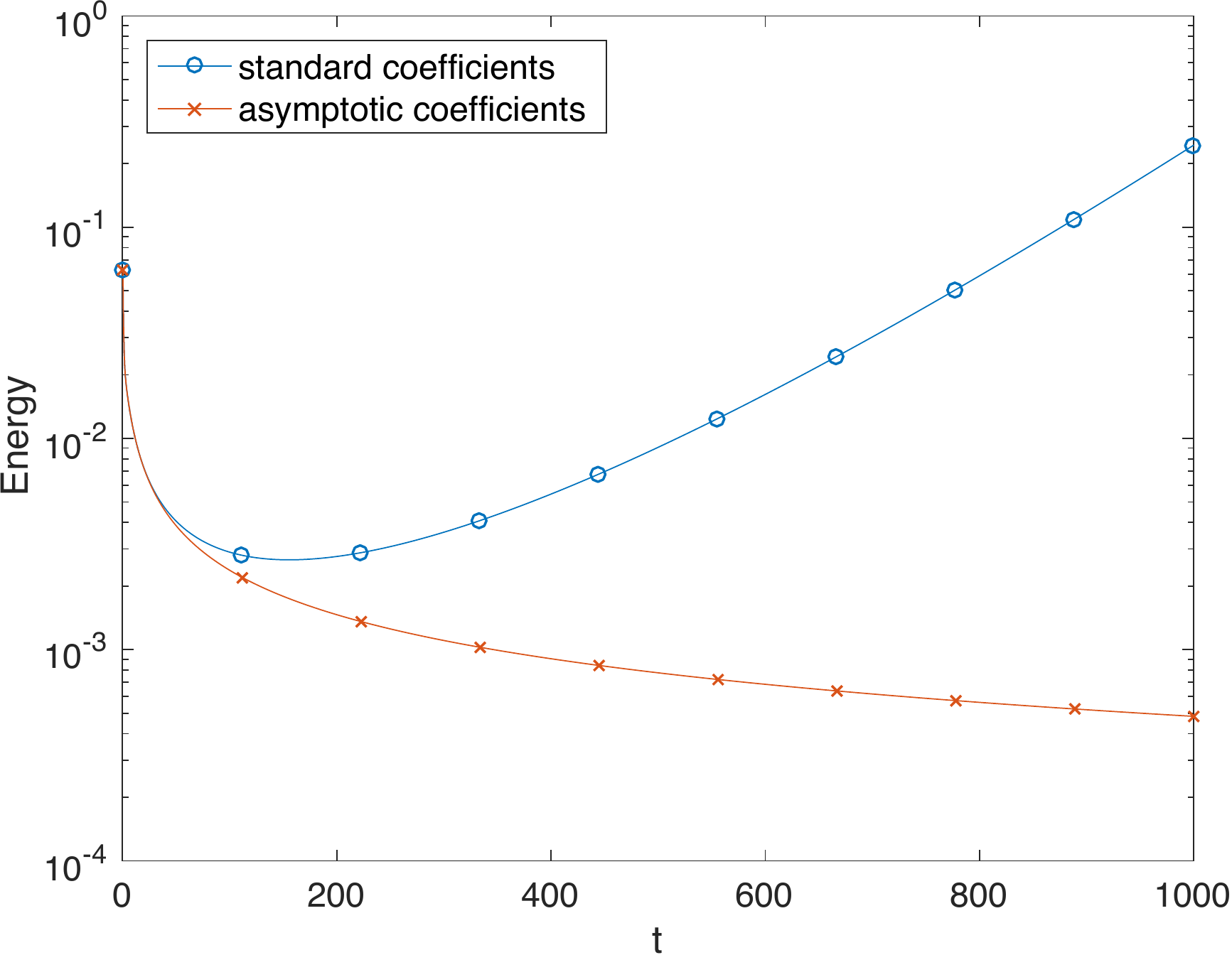} 
\end{center}  
  \caption{Evolution of the discrete energy $\mathcal{E}_n$ of the solution with standard and asymptotic convolution coefficients.}
  \label{fig:evol02}
\end{figure}

\section{Conclusion}

In this paper, we derived continuous and discrete transparent boundary conditions for the linearized mixed (KdV)-(BBM) equation. We chose finite difference centered scheme for spatial derivatives and a Crank Nicolson scheme in time to achieve second order in time and space and to preserve some invariants in the equation (spatial mean, energy). Continuous transparent boundary conditions are proved to be stable whereas we provide sufficient conditions in the discrete case. Moreover the discrete transparent boundary conditions are proved to be consistent with the continuous ones.

From a numerical view point, the key step is to compute the inverse $\mathcal{Z}$-transform of convolution kernels. We propose a new strategy based on the fact that convolution kernels are products and sums of roots of some characteristic polynomial: we simply compute an asymptotic expansion of these roots as $x=z^{-1}\to 0$ where $z$ is the parameter in $\mathcal{Z}$-transform. This method is proved to be very efficient and stable except for small $\delta x$. Here, we propose an alternative strategy based on an asymptotic expansion of convolution kernel with respect to $\delta x$. We show that the resulting coefficients have a good behavior for large time simulation which is not the case for the first strategy.

In practice, we will have to deal with non-linear equations. In order to derive transparent boundary conditions in the nonlinear case, we will adapt our strategy to linear equations with variable coefficients
and then adopt a fixed point strategy: see \cite{AABES} for more details in the case of
nonlinear Schrodinger equations. We shall use this strategy to study accurately the
interaction of solitons in BBM equations like \cite{E-McG}, \cite{DP} (where non physical boundary conditions were used). 

Other interesting questions concerns the design of discrete transparent boundary conditions for more general models of water waves. On the one hand, it would be of interest to adapt this strategy to two dimensional models for large wavelength weakly nonlinear water waves like the Kadomtsev-Petviashvili (KP) equation: the main issue there is to deal with non local terms in the equation. A close model is also the Zakharov-Kuznetsov equation \cite{ZE}. On the other hand, it would be of interest  to derive discrete transparent boundary conditions in the case of the Serre-Green-Naghdi equations \cite{L} which are physically more relevant for the water wave problem: the main issue there is to design discrete transparent boundary conditions in the context of systems of partial differential equations instead of scalar partial differential equations.\\[3mm]

\noindent \textbf{Acknowledgments} The authors would like to thank Jean-Fran\c{c}ois Coulombel for stimulating and valuable discussions related to this work. Research of C. Besse was partially supported by the French ANR project BoND ANR-13-BS01-0009-01 and by the French ANR project MOONRISE ANR-14-CE23-0007-01. Research of P. Noble was partially supported by the ANR project French BoND ANR-13-BS01-0009-01.

\bibliographystyle{abbrv}
\bibliography{biblio}

\section{Appendix}

We present the limit $\delta x\to 0$ for the general case KdV-BBM of the convolution coefficients. It extends the case of the (lKdV) equation performed in subsection \ref{asympt_cof}. The roots leading to the convolution coefficients are solution to
\[
\begin{array}{lcl}
P(r) & = & \dsp r^4 - \left ( 2 - a + \mu p(z) \right ) r^3 + \left ( \frac{4a}{\lambda_H}+2 \mu \right ) p(z) r^2 + \left ( 2 - a - \mu p(z) \right ) r - 1, \\
& = & \dsp r^4 - \left (2- \frac{c \delta x^2}{\varepsilon} + \frac{4 \alpha \delta x}{\varepsilon \delta t} p(z) \right ) r^3 + \left (\frac{4 \delta x^3}{\varepsilon \delta t} + \frac{8 \alpha \delta x}{\varepsilon \delta t} \right ) p(z) r^2 + \left ( 2 - \frac{c \delta x^2}{\varepsilon} - \frac{4 \alpha \delta x}{\varepsilon \delta t} p(z) \right )r-1,\\
& = & \left(r^2-s^s(z) r+p^s(z)\right)\left(r^2- s^u(z) r+ p^u(z)\right).
\end{array}
\]
\noindent
The roots of $r^2 - s^s r + p^s$ belongs to $\{ r \in \mathbb{C}, \ |r|<1\}$ whereas the ones of $r^2 - s^u r + p^u$ belongs to $\{ r \in \mathbb{C}, \ |r|>1\}$.
Let us calculate $(s^s,p^s,s^u,p^u)$. These functions satisfy
\begin{equation}\label{eqvp}
\left \{
\begin{array}{lcl}
s^s + s^u & = & \dsp  2 + \frac{4 \alpha p(z)}{\varepsilon \delta t} \dx -\frac{c}{\varepsilon}\dx^2, \\
s^s s^u + p^s +p^u & = & \dsp \frac{8 \alpha p(z)}{\varepsilon \delta t} \dx + \frac{4 p(z)}{\varepsilon \delta t} \dx^3, \\
s^s p^u + s^u p^s & = & \dsp -2 + \frac{4 \alpha p(z)}{\varepsilon \delta t} \dx +\frac{c}{\varepsilon}\dx^2,\\
p^s p^u & = & -1.
\end{array}
\right.
\end{equation}
We look for an asymptotic expansion of these quantities as $\dx \to 0$ in the form:
\[
s^s = \sum_{k\ge 0} s_k \dx^{k}, \ \ p^s = \sum_{k\ge 0} p_k \dx^{k},\ \ s^u = \sum_{k\ge 0} t_k \dx^{k}, \ \ p^u = \sum_{k\ge 0} q_k \dx^{k}.
\]
By inserting this  expansion into (\ref{eqvp}) and identifying $O(\delta x^p)$ terms with $(p\in\mathbb{N})$, we obtain a non linear system and a serie of linear systems to be solved. First, by identifying {\bf $0^\text{th}$ order terms}, one finds the nonlinear system of equations:
\[
\left \{
\begin{array}{lcl}
s_0 +t_0 & = & 2, \\
s_0t_0+p_0+ q_0 & = & 0, \\
s_0q_0+t_0p_0 & = & -2,\\
p_0q_0 & = & -1.
\end{array}
\right.
\]
The solution writes $(s_0,p_0,t_0,q_0) = (0,-1,2,1)$. Next, we identify $O(\delta x^p)$ terms with $p\geq 1$. One finds the family of linear systems:

\[
A \begin{pmatrix}
s_n \\ p_n \\ t_n \\q_n
\end{pmatrix} = F_n =  \Sigma_n - G_n \ \text{ where } A=\begin{pmatrix}
1 & 0 & 1 & 0 \\ t_0 & 1 & s_0 & 1 \\ q_0 & t_0 & p_0 & s_0 \\ 0 & q_0 & 0 & p_0
\end{pmatrix} = \begin{pmatrix}
1 & 0 & 1 & 0 \\ 2 & 1 & 0 &  1 \\ 1 & 2 & -1 & 0 \\ 0 & 1 & 0 & -1
\end{pmatrix},
\]
\[
\begin{array}{l}
\Sigma_1 = \begin{pmatrix}
\dsp \frac{4 \alpha p}{\varepsilon \delta t} \\ \dsp \frac{8 \alpha p}{\varepsilon \delta t} \\ \dsp \frac{4 \alpha p}{\varepsilon \delta t} \\ 0 
\end{pmatrix}, \ 
\Sigma_2 = \begin{pmatrix}
\dsp - \frac{c}{\varepsilon} \\ 0 \\ \dsp  \frac{c}{\varepsilon} \\ 0 
\end{pmatrix},
\ \Sigma_3 = \begin{pmatrix}
0 \\ \dsp \frac{4}{\varepsilon \delta t}p(z)  \\ 0 \\0
\end{pmatrix}, \ \Sigma_n = \begin{pmatrix}
0 \\ 0 \\ 0 \\0
\end{pmatrix} \text{ if } n \ge 4, \\
G_n = \begin{pmatrix}
0 \\ \sum_{k=1}^{n-1}s_k t_{n-k} \\ \sum_{k=1}^{n-1} s_k q_{n-k} + \sum_{k=1}^{n-1} t_k p_{n-k} \\ \sum_{k=1}^{n-1} p_k q_{n-k}
\end{pmatrix}.
\end{array}
\]
The matrix $A$ is not invertible, the eigenvalue $0$ is simple and associated to $v=\begin{pmatrix}1 \\ -1 \\ - 1 \\ -1 \end{pmatrix}$. If the condition 
$$
\displaystyle
\det\left (F_n,\begin{pmatrix} 0 \\ 1 \\ 2 \\1 \end{pmatrix}, \begin{pmatrix} 1 \\ 0 \\ -1 \\ 0 \end{pmatrix}, \begin{pmatrix}
0 \\ 1 \\ 0 \\ -1 \end{pmatrix} \right )=0$$ 
is fulfilled, then $U_n =(s_n,p_n,t_n,q_n)^T$ is given by
\[
U_n = \lambda_n v + \frac{(F_n)_2+(F_n)_4}{2} e_2 + \left ((F_n)_2-(F_n)_3+(F_n)_4 \right ) e_3+ \frac{F_2-F_4}{2}e_4 = \left (\begin{array}{rcl}
\lambda_n & & \\ - \lambda_n & +& \dsp \frac{(F_n)_2+(F_n)_4}{2} \\ -\lambda_n & +& \dsp (F_n)_2-(F_n)_3+(F_n)_4 \\ - \lambda_n & + & \dsp \frac{F_2-F_4}{2}
\end{array}\right ),
\]
where $(e_1,e_2,e_3,e_4)$ is the canonic basis of $\R^4$. Let $\lambda_1$ the root of 
\[
\lambda_1^3 - \frac{8 \alpha p}{\varepsilon \delta t} \lambda_1^2 + \left ( \frac{c}{\varepsilon} + \frac{20 \alpha^2 p^2}{\varepsilon^2 \delta t^2} \right ) \lambda_1 + \left ( \frac{2p}{\varepsilon \delta t} - \frac{2 \alpha c p}{\epsilon^2 \delta t} - \frac{16 \alpha^3 p^3}{\varepsilon^3 \delta t^3} \right ) = 0.
\]
whose real part is negative. We get:
\[
s^s  = \dsp  \lambda_1  \dx + a_2 \dx^2 +  O(\dx^3),
\]
\[
s^u = 2 + \left (\frac{4 \alpha p}{\varepsilon \delta t} - \lambda_1 \right ) \dx -  \left ( a_2 + \frac{c}{\varepsilon}  \right )  \dx^2  + O(\dx^3) ,
\]
\[
p^s =  -1 + \left (\frac{4 \alpha p}{\varepsilon \delta t} - \lambda_1 \right ) \dx - \left  (a_2 - \frac{2 \alpha \lambda_1 p}{\varepsilon \delta t} + \frac{8\alpha^2 p^2}{\varepsilon^2 \delta t^2} \right ) \dx^2 +  O(\dx^3),
\]
\[
p^u =  1 + \left (\frac{4 \alpha p}{\varepsilon \delta t} - \lambda_1 \right ) \dx + \left ( \lambda_1^2 - \frac{6 \alpha \lambda_1 p}{\varepsilon \delta t} + \frac{8 \alpha^2 p^2}{\varepsilon^2 \delta t^2} - a_2 \right ) \dx^2  +  O(\dx^3),
\]
where 
\[
a_2 = - \frac{1}{2} \frac{\alpha^4 \varepsilon u^4 - 3 \lambda_1 \alpha^3 \varepsilon u^3 + 2 \lambda_1^2 \alpha^2 \varepsilon u^2 + 2 \alpha^2 c u^2 - 6 \lambda_1 \alpha c u - 2 \alpha \varepsilon u^2+ 8 c \lambda_1^2+6 \lambda_1 \varepsilon u }{4c +12 \varepsilon \lambda_1^2 -16 \alpha \varepsilon u \lambda_1 + 5 \alpha^2 \varepsilon u^2}
\]
and $u = \frac{4p}{\varepsilon \delta t}$.

The asymptotic expansions of various terms are\[
\begin{array}{lcl}
P & = & \dsp \frac{c}{\varepsilon} - \frac{\alpha^2 u^2}{12}, \\
& = &  \dsp \left ( \frac{c}{\varepsilon} - \frac{4 \alpha^2}{3 \varepsilon^2 \delta t^2} \right ) - \frac{4 \alpha^2}{3 \varepsilon^2 \delta t^2} \left ( \beta^{(-2)}_1 - 2 \beta^{(-2)}_0 \right ) \frac{1}{z} - \frac{4 \alpha^2}{3 \varepsilon^2 \delta t^2} \sum_{l \ge 2} \frac{1}{z^l} \left (\beta^{(-2)}_l - 2 \beta^{(-2)}_{l-1} + \beta^{(-2)}_{l-2} \right ), \\
& =&   \dsp \sum_{l \ge 0} \frac{P_l}{z^l}, \\
A &  = & \dsp  \frac{2 \alpha c}{3 \varepsilon^2 \delta t} + \frac{2}{\varepsilon \delta t} ,\\
B & = & \dsp \frac{16\times 19 \times 5 \alpha^3}{27  \varepsilon^3 \delta t^3}, \\
Q & =  & \dsp \left ( A \beta^{(-1)}_0 + B \beta^{(-3)}_0 \right ) + \frac{1}{z} \left ( A \beta^{(-1)}_1 - A \beta^{(-1)}_0 + B\beta^{(-3)}_1 -3 B \beta^{(-3)}_0 \right ) \\
& & \dsp + \frac{1}{z^2} \left ( A (\beta^{(-1)}_2 - \beta^{(-1)}_1 ) + B (\beta^{(-3)}_2 - 3 \beta^{(-3)}_1 +3 \beta^{(-3)}_0 \right ) \\ 
& & + \dsp \sum_{l \ge 3} \frac{1}{z^l} \left [ A \left ( \beta^{(-1)}_l - \beta^{(-1)}_{l-1} \right ) + B \left ( \beta^{(-3)}_l - 3 \beta^{(-3)}_{l-1} +3 \beta^{(-3)}_{l-2} - \beta^{(-3)}_{l-3} \right ) \right ], \\
& = & \dsp \sum_{l\ge 0} \frac{Q_l}{z^l} ,\\
\Delta & = & \dsp Q^2 + \frac{4}{27} P^3, \\
& = & \dsp \sum_{l \ge 0} \frac{1}{z^l} \left ( \sum_{k=0}^l Q_kQ_{l-k} + \frac{4}{27} \sum_{k_1=0}^l \sum_{k_2=0}^{l-k_1} P_{k_1} P_{k_2} P_{l- k_1 - k_2} \right ), \\
& = & \dsp \sum_{l \ge 0 } \frac{\Delta_l}{z^l}, \\
\delta & = & \dsp \Delta^{1/2,} \\
& = & \dsp  \beta_0^{(1/2)} \Delta_0^{1/2} + \sum_{l \ge 1} \frac{1}{z^l} \frac{1}{\Delta_0^{l-1/2}} \sum_{j=1}^l \beta_j^{(1/2)} \left ( \sum_{(k_1,\dots,k_j) \in \{1,\dots,l\}^j, \sum_{i=1}^j k_j = l} \prod_{i=1}^j \Delta_{k_i} \right ), \\
& = & \dsp  \sum_{l \ge 0} \frac{\delta_l}{z^l}, \\
\zeta & = & \dsp \frac{1}{2} (-Q + \delta) = \sum_{l \ge 0} \frac{1}{z^l} \left (- \frac{1}{2} q_l + \frac{1}{2} \delta_l \right ) = \sum_{l \ge 0} \frac{\zeta_l}{z^l},
\end{array}
\]
\[
\begin{array}{lcl}
\zeta^{1/3} & = & \dsp \zeta_0^{1/3} + \sum_{l \ge 1} \frac{1}{z^l} \frac{1}{\zeta_0^{l-1/3}} \sum_{j=1}^l \beta^{(1/3)}_j \left ( \sum_{(k_1,\dots,k_j) \in \{1,\dots,l\}^j , \sum_{i=1}^j k_i = l} \prod_{i=1}^j \zeta_{k_i} \right ) = \sum_{l \ge 0} \frac{\mu_l^{(1/3)}}{z^l} ,\\
\zeta^{-1/3} & = & \dsp \zeta_0^{-1/3} + \sum_{l \ge 1} \frac{1}{z^l} \frac{1}{\zeta_0^{l+1/3}} \sum_{j=1}^l \beta^{(-1/3)}_j \left ( \sum_{(k_1,\dots,k_j) \in \{1,\dots,l\}^j , \sum_{i=1}^j k_i = l} \prod_{i=1}^j \zeta_{k_i} \right ) = \sum_{l \ge 0} \frac{\mu_l^{(-1/3)}}{z^l} ,\\
\lambda_k & = & \dsp \frac{8 \alpha}{3 \varepsilon \delta t} \beta_0^{(-1)} + j^{k-1} \mu_0^{(1/3)} - \frac{1}{3 j^{k-1}} P_0 \mu^{(-1/3)}_0 ,\\
& & \dsp + \sum_{l \ge 1} \frac{1}{z^l} \left ( \frac{8 \alpha}{3 \varepsilon \delta t } \left ( \beta^{(-1)}_l- \beta^{(-1)}_{l-1} \right ) + j^{k-1} \mu_l^{(1/3)} - \frac{1}{3 j^{k-1}} \sum_{j=0}^l P_j \mu^{(-1/3)}_{l-j} \right ), \\
& = &\dsp \sum_{l \ge 0} \frac{\lambda_{k,l}}{z^l}.
\end{array}
\]

Finally, it leads to the asymptotic coefficients
\[
\begin{array}{lcl}
\widetilde{as}^s  & = & \dsp  \lambda_1 \left ( 1 + \frac{1}{z} \right ) \dx +   O(\dx^2), \\
& = & \dsp \left ( \lambda_{1,0} \dx + O(\dx^2) \right ) + \sum_{l \ge 1} \frac{1}{z^l} \left ( (\lambda_{1,l} + \lambda_{1,l-1}) \dx + O(\dx^2)  \right ),
\end{array}
\]
\[
\begin{array}{lcl}
\widetilde{as}^u & = & \dsp 2 \left (1 + \frac{1}{z} \right ) + \left (1 + \frac{1}{z} \right )  \left (\frac{4 \alpha p}{\varepsilon \delta t} - \lambda_1 \right ) \dx  + O(\dx^2), \\
&= & \dsp \left (2 + \left (\frac{4 \alpha}{\varepsilon \delta t} - \lambda_{1,0} \right ) \dx  + O(\dx^2) \right ) + \frac{1}{z} \left ( 2 - \left ( \frac{4 \alpha}{\varepsilon \delta t} + \lambda_{1,1} + \lambda_{1,0} \right ) \dx + O(\dx^2) \right ) \\
& & \dsp + \sum_{l \ge 2} \frac{1}{z^l} \left ( \left (- \lambda_{1,l} - \lambda_{1,l-1} \right ) \dx + O(\dx^2) \right ),
\end{array}
\]
\[
\begin{array}{lcl}
\widetilde{ap}^s & = & \dsp - \left ( 1+ \frac{1}{z} \right )  + \left ( 1+ \frac{1}{z} \right )\left (\frac{4 \alpha p}{\varepsilon \delta t} - \lambda_1 \right ) \dx + O(\dx^2), \\
& = & \dsp \left ( -1 + \left (\frac{4 \alpha}{\varepsilon \delta t} - \lambda_{1,0} \right ) \dx  + O(\dx^2)  \right ) + \frac{1}{z} \left ( -1 - \left ( \frac{4 \alpha}{\varepsilon \delta t} + \lambda_{1,1} + \lambda_{1,0} \right ) \dx + O(\dx^2) \right ) \\
& & \dsp + \sum_{l \ge 2} \frac{1}{z^l} \left ( \left (- \lambda_{1,l} - \lambda_{1,l-1} \right ) \dx + O(\dx^2) \right ),
\end{array}
\]
\[
\begin{array}{lcl}
\widetilde{ap}^u & = & \dsp \left (1+ \frac{1}{z} \right ) + \left ( 1 + \frac{1}{z} \right ) \left (\frac{4 \alpha p}{\varepsilon \delta t} - \lambda_1 \right ) \dx +  O(\dx^2), \\
& = & \dsp \left ( 1 + \left (\frac{4 \alpha}{\varepsilon \delta t} - \lambda_{1,0} \right ) \dx  + O(\dx^2)  \right ) + \frac{1}{z} \left ( 1 - \left ( \frac{4 \alpha}{\varepsilon \delta t} + \lambda_{1,1} + \lambda_{1,0} \right ) \dx + O(\dx^2) \right ) \\
& & \dsp + \sum_{l \ge 2} \frac{1}{z^l} \left ( \left (- \lambda_{1,l} - \lambda_{1,l-1} \right ) \dx + O(\dx^2) \right ).
\end{array}
\]

\end{document}